%% file: main.tex
\documentclass[11pt,letterpaper]{amsart}
\usepackage[margin=1.1in]{geometry}
\usepackage{amsrefs}
\usepackage[utf8]{inputenc}
\usepackage[all]{xy}
\input{preamble}
\makeatletter
\newcommand{\xdashrightarrow}[2][]{\ext@arrow 0359\rightarrowfill@@{#1}{#2}}
\def\rightarrowfill@@{\arrowfill@@\relax\relbar\rightarrow}
\def\arrowfill@@#1#2#3#4{%
	$\m@th\thickmuskip0mu\medmuskip\thickmuskip\thinmuskip\thickmuskip
	\relax#4#1
	\xleaders\hbox{$#4#2$}\hfill
	#3$%
}
\makeatother
\usepackage{hyperref}
\newcommand{\PP}{\mathbb{P}}

\newcommand{\tC}{\tilde{C}}
\newcommand{\kbar}{\overline{k}}
\newcommand{\ksep}{k^\mathrm{sep}}
\newcommand{\bs}{\cv}
\DeclareMathOperator{\Aut}{Aut}

\DeclareMathOperator{\Jac}{Jac}
\DeclareMathOperator{\Prym}{Prym}
\DeclareMathOperator{\Norm}{Norm}
\DeclareMathOperator{\Trace}{Trace}
\DeclareMathOperator{\Gal}{Gal}

\DeclareMathOperator{\kum}{Kum}
\title{Prym varieties of genus four curves}
\author{Nils Bruin}
\address{Nils Bruin\\ Simon Fraser University \\
8888 University Drive \\
Burnaby, British Columbia, V5A 1S6, Canada}
\email{nbruin@sfu.ca}
\author{Emre Can Sert\"oz}
\address{Emre Can Sert\"oz\\
Max Planck Institute for Mathematics in the Sciences, Inselstr. 22\\ 04103 Leipzig, Germany}
\email{emresertoz@gmail.com}
\thanks{Research of the first author partially supported by NSERC}

\begin{document}

\begin{abstract}
Double covers of a generic genus four curve C are in bijection with Cayley cubics containing the canonical model of C. The Prym variety associated to a double cover is a quadratic twist of the Jacobian of a genus three curve X. The curve X can be obtained by intersecting the dual of the corresponding Cayley cubic with the dual of the quadric containing C. We take this construction to its limit, studying all smooth degenerations and proving that the construction, with appropriate modifications, extends to the complement of a specific divisor in moduli. We work over an arbitrary field of characteristic different from two in order to facilitate arithmetic applications.
\end{abstract}

\maketitle

\section{Introduction}

Let $C \subset \ppp_{\Cc}$ be a canonically embedded complex curve of genus four. If $C$ is generic, there is a bijection between the $255$ two-torsion points $\varepsilon \in \jac_C[2]$ of the Jacobian of $C$ and Cayley cubics $\Gamma_\varepsilon\subset \ppp$ containing $C$ (see~\cite{catanese83}). A Cayley cubic is a surface of degree three with four simple nodes and its strict projective dual $\vGam_\varepsilon \subset \vppp$ is a Steiner quartic. Let $Q_C$ be the quadric containing $C$ and $\vQ_C$ its dual, then $\vGam_\varepsilon \cap \vQ_C$ is a singular model of a genus $3$ curve $X_\varepsilon$. A beautiful classical story culminates with the observation that the Jacobian of $X_\varepsilon$ is the Prym variety $\Prym(\tilde C_\varepsilon/C)$ (see~\cites{abelian-varieties, mumford--prym}) associated to the unramified double cover $\tilde C_\varepsilon \to C$ induced by $\varepsilon$ (see~\cites{Milne1923,coble--theta-book,Recillas1993}).

In the present paper we generalize this classical work in two directions. For one, we remove the genericity assumptions on $(C,\varepsilon)$ and generalize the construction above to its natural limit. In addition, with an eye towards arithmetic applications to Prym varieties~\cites{Bruin2008,BruinFlynn2005}, we perform these constructions over an arbitrary base field of characteristic not two.

Milne~\cite{Milne1923} was the first to observe, from a projective geometric setting, that there is a natural bijection between certain pairs of tritangents of $C$ and $28$ bitangents of a plane quartic. Milne's paper served as an inspiration for our work and we dedicate the last section to retelling this beautiful story in modern language. 

\subsection{Construction}

Let $k$ be a field of characteristic not two, let $C/k$ be a proper smooth curve of genus four and $\varepsilon \in \jac_C[2](k)$ a two-torsion point. 

We need to distinguish between the special pairs $(C,\varepsilon)$ for which the construction of $X_\varepsilon$ behaves differently. If $Q_C$ is singular, then $C$ admits a vanishing theta-null $\theta$, which induces a natural bijection between two torsion elements $\varepsilon$ in $\jac_C[2]$ and theta characteristics $\theta\otimes \varepsilon$ on $C$. In this case, we call $\varepsilon$ \emph{even} or \emph{odd} according to the parity of the theta characteristic $\theta\otimes \varepsilon$. Whether or not $Q_C$ is singular, if $C$ admits a degree two map to a genus one curve $E$ and $\varepsilon$ is the pull-back of a line bundle on $E$ we say $\varepsilon$ is \emph{bielliptic}.

A uniform treatment for these cases can be given using the notion of a \emph{symmetrization}. A cubic symmetroid $\Gamma \subset \ppp$ is, by definition, cut out by the determinant of a $3 \times 3$ symmetric matrix of linear forms defined over $\kbar$. Equivalently, $\ppp$ parametrizes conics on $\pp$ in such a way that $\Gamma$ becomes the locus of singular conics. Evaluating the conics parametrized by $\ppp$ we get the \emph{symmetrization} of $\Gamma$, a map $\xq\colon \pp \to \vppp$ into the dual space of $\ppp$.

We remove all genericity assumptions from~\cite{catanese83}*{Theorem 1.5} and obtain the following result in Section~\ref{sec:prym_canonical_map}, which is key for the rest of our results.

\begin{theorem}\label{thm:prym_from_intro}
  Every line bundle $\varepsilon$ of order two on $C$ gives rise to a natural construction of a cubic symmetroid $\Gamma_\varepsilon \subset \ppp$ containing $C$ as well as a symmetrization $\xq_\varepsilon\colon \pp \to \vppp$ of $\Gamma_\varepsilon$. This construction can be reversed to recover the class of $\varepsilon$ from the pair $(\Gamma_\varepsilon,\xq_\varepsilon)$.
\end{theorem}

\begin{remark}
  If $\varepsilon$ is bielliptic then $\Gamma_\varepsilon$ is \emph{degenerate}, meaning it is a cone over a cubic curve. When $\varepsilon$ is odd then $\Gamma_\varepsilon$ is \emph{reducible}. Otherwise, $\Gamma_\varepsilon$ is irreducible and normal. 
\end{remark}

\begin{theorem}\label{thm:intro_X}
  If $\varepsilon$ is not odd, then the image of $X_\varepsilon$ under its canonical map is the pullback $\xq_\varepsilon\inv(\vQ_C)$ of the dual quadric $\vQ_C$ by the symmetrization $\xq_\varepsilon$. 
\end{theorem}

In Section~\ref{sec:prym_construction} we show how this construction induces on $X_\varepsilon$ an unordered pair $\{\cl_1,\cl_2\}$ of residual tetragonal pencils, i.e., $\cl_1\otimes \cl_2 \simeq \omega_{X_\varepsilon}^{\otimes 2}$. Consequently, we obtain a $k$-valued point $\kappa_\varepsilon=\{\cl_1\otimes \omega_X^\vee, \cl_2 \otimes \omega_X^\vee\}$ of the Kummer variety $\kum_X=\jac_X/\langle -1 \rangle$ of $X$. We explain how to reverse this construction and obtain $(C,\varepsilon)$ from $(X_\varepsilon,\kappa_\varepsilon)$ in the next section.

These constructions can be carried out over arbitrary $k$ primarily because of the following result, proved in Section~\ref{sec:cubic_classification_over_k}. An implication of this result is Proposition~\ref{prop:descend_epsilon} which states that if the isomorphism class of $\varepsilon$ is defined over $k$ then there is a representative line bundle over $k$.
 
\begin{proposition}\label{prop:intro_k_symmetrization}
  A non-degenerate cubic symmetroid admits a symmetrization over its field of definition.
\end{proposition}

\subsection{Reverse construction}\label{sec:reverse_from_intro}

Suppose $X$ is a curve of genus three and $\kappa=\{\cn,\cn^\vee\}$ is a non-zero, $k$-valued point of the Kummer variety $\kum_{X}=\Jac_{X}/\langle -1\rangle$. Over $\kbar$ we define the pair of tetragonal pencils $\cl_1=\omega_X\otimes \cn$ and $\cl_2=\omega_X\otimes \cn^\vee$.

Only the unordered pair $\{\cl_1,\cl_2\}$ is guaranteed to be defined over $k$, but for ease of exposition we assume here that the $\cl_i$ are defined over $k$ individually, leaving the general case for Proposition~\ref{P:QGammadescent}. Let $W_i=\H^0(X,\cl_i)$ and, with $j_1,\, j_2$ denoting multiplication maps, define a four-dimensional vector space $\MUV$ as follows:
\begin{equation}\label{eq:reverse--M}
  \begin{tikzcd}
    \MUV \arrow[dr, phantom, "\ulcorner", very near start] \arrow[r]\arrow[d] & W_1\otimes W_2 \arrow[d,"j_2"]\\
    \sym^2 \H^0(\omega_X) \arrow[r,"j_1"'] & \H^0(\omega_X^{\otimes 2}).
  \end{tikzcd}
\end{equation}

The space $|\sym^2 \H^0(\omega_X)|$ parametrizes conics on $\pp = \Pp\H^0(\omega_X)$ and therefore contains the cubic discriminant locus $\cd$ parametrizing singular conics. Let $\Gamma_{X,\kappa}$ be the cubic symmetroid in $|\MUV|$ obtained by pulling back $\cd$, which comes with a symmetrization $\xq_{X,\kappa}$.  Denote the pullback of the smooth quadric $|W_1| \times |W_2| \toi |W_1 \otimes W_2|$ into $|\MUV|$ by $Q_{X,\kappa}$. 

\begin{theorem}\label{thm:reconstruction}
  If $C=Q_{X,\kappa}\cap \Gamma_{X,\kappa}\subset |\MUV|$ is a curve of geometric genus four, then the canonical isomorphism $|\MUV|\simeq \Pp \H^0(\omega_C)$ identifies the pair $(\Gamma_{X,\kappa},\xq_{X,\kappa})$ with a pair $(\Gamma_\varepsilon,\xq_\varepsilon)$ as in Theorem~\ref{thm:prym_from_intro}, for which $\Prym(\tC_\varepsilon/C)=\Jac_{X}$.
\end{theorem}

\subsection{Stratification}

The results in the preceding theorems can be tied together as follows. Let $\cR_4$ be the moduli space of pairs $(C,\varepsilon)$, where $C$ is a curve of genus $4$ and $\varepsilon\in\Pic(C)[2] \setminus \{0\}$. We define the following stratification of the open locus in $\cR_4$ where $C$ is not hyperelliptic:

\begin{center}
\begin{tabular}{ll}
  $\cR_4^\mathrm{o}$          & non-hyperelliptic curves excluding the special cases below, \\
  $\cR_4^\mathrm{odd}$        & curves with vanishing theta-null and $\varepsilon$ odd,  \\
  $\cR_4^\mathrm{even}$       & curves with vanishing theta-null and $\varepsilon$ even,  \\
  $\cR_4^\mathrm{biell}$      & bielliptic $(C,\varepsilon)$ without vanishing theta-null, \\
  $\cR_4^\mathrm{biell,even}$ & bielliptic $(C,\varepsilon)$ with vanishing theta-null.
\end{tabular}
\end{center}
Corollary~\ref{cor:odd_bielliptic} shows that $\varepsilon$ cannot be bielliptic \emph{and} odd.

The curve $X_\varepsilon$ may be hyperelliptic and the pencils $\{\cl_1,\cl_2\}$ can be canonical ($\cl_i\simeq\omega_{X_\varepsilon}$), self-residual $(\cl_1\simeq\cl_2)$ or neither (general). The strata above reflect these properties.
\begin{table}
\[
\begin{array}{c||c|c||c|c}
  \text{stratum} & Q_C & \Gamma_\varepsilon & X_\varepsilon & \{\cl_1,\cl_2\}\\
\hline
\cR_4^\mathrm{o}          & \text{nonsingular} & \text{irreducible} & \text{non-hyperelliptic} & \text{general}\\
\cR_4^\mathrm{even}       & \text{cone}        & \text{irreducible} & \text{hyperelliptic}      & \text{general}\\
\cR_4^\mathrm{biell}      & \text{nonsingular} & \text{degenerate}  & \text{non-hyperelliptic} & \text{self-residual}\\
\cR_4^\mathrm{biell,even} & \text{cone}        & \text{degenerate}  & \text{hyperelliptic}      & \text{self-residual} \\
\hline
\cR_4^\mathrm{odd}        & \text{cone}        & \text{reducible}   & \text{non-hyperelliptic} & \text{canonical}
\end{array}
\]
\caption{Strata of $\cR_4$, correlated with properties of $C$ and $X$}
\label{tbl:strata} 
\end{table}
\begin{remark}
  If $(C,\varepsilon)$ is odd, we cannot expect to construct the curve $X_\varepsilon$ from $(\Gamma_\varepsilon,\xq_\varepsilon)$ since $C$ itself is no longer determined by $\Gamma_\varepsilon$; indeed $\Gamma_\varepsilon \cap Q_C = Q_C$, see Section~\ref{sec:prym_canonical_map}. The odd case allows for another construction, see Section~\ref{sec:odd_del_pezzo}.
\end{remark}

\begin{remark}
  When $k$ is not algebraically closed, the double cover $\tC_\varepsilon\to C$ may admit quadratic twists. In this case, the symmetrization of $\Gamma_\varepsilon$ distinguishes a twist and, therefore, fully determines $X_\varepsilon$. In general, even without a symmetrization, one can find the correct twist of $\tC_\varepsilon \to C$ as it will be the one yielding a Prym variety that is also the Jacobian of a curve. But when $X_\varepsilon$ is hyperelliptic, the symmetrization becomes crucial since twists of the Jacobian of $X_\varepsilon$ are also Jacobians, hence every twist of $\tC_\varepsilon \to C$ yields a Prym variety that is also a Jacobian.
    
  For the reverse construction there is another arithmetic subtlety. Starting with a self-residual pencil on $X$, one obtains a bielliptic $C \to E$ over $k$, but no particular quadratic twist of this cover is distinguished by data on $X$. This reflects that quadratic twists of a given $C\to E$ yield the same $(X_\varepsilon,\cl,\cl)$.
\end{remark}

\subsection{Tritangents and bitangents}

In a paper from 1923, Milne~\cite{Milne1923} constructs 255 plane quartic curves associated to a general canonical curve $C$ of genus four. Each of the $28\cdot 255$ bitangents of these quartics are shown to be in bijection with the unordered pairs of tritangents of $C$ via an explicit construction. This paper of Milne was our source of inspiration and we give a modern treatment of his work in Section~\ref{sec:classical}, showing that his genus three curves indeed give rise to the corresponding Prym varieties. 

Milne specifies the line bundle $\varepsilon$ by giving a set of six points on the canonical curve $C \subset \ppp$ supporting a generic divisor $D \in |\omega_C \otimes \varepsilon|$. He then recovers $\Gamma_\varepsilon$ via the following.

\begin{proposition}
 Suppose $(C,\varepsilon)$ is neither odd nor bielliptic. If $D \in |\omega_C \otimes \varepsilon|$ is generic, then there exists a twisted cubic $T$ passing through the six points in the support of $D \subset C \subset \ppp$. There is a unique cubic hypersurface containing $C$ and $T$---it is the Cayley cubic $\Gamma_\varepsilon$.
\end{proposition}

Suppose $H_1,\,H_2 \subset \ppp$ are two tritangents of $C$ with contact divisors $D_1,\, D_2$ so that $H_i \cdot C = 2D_i$ and $D_1 + D_2 \in |\omega_C \otimes \varepsilon|$. If $C$ is generic then $D_1+D_2$ satisfies the genericity assumption of the proposition above. Milne's observation rests on the following fact.

\begin{theorem}
  Let $T \subset \Gamma_\varepsilon$ be the twisted cubic passing through $D_1+D_2$. As a subset of $\Gamma_\varepsilon$, $T$ parametrizes singular conics in $\pp$. The singularities of these conics trace a bitangent of $X_\varepsilon$. 
\end{theorem}

\subsection{Review of literature}\label{sec:review}

Our work is a geometric realization of Donagi's~\cite{donagi} extension of Recillas' trigonal construction~\cite{Recillas1974}. Consider the moduli space $\cR_4^\mathrm{trig}$ of triplets $(C,\varepsilon,\varphi\colon C \overset{3:1}{\longrightarrow} \p)$ where $(C,\varepsilon) \in \cR_4$ and $\varphi$ is a trigonal pencil. Let $\cm_3^\mathrm{tetr}$ be the moduli of pairs $(X,\psi\colon X \overset{4:1}{\longrightarrow}\p)$ where $X$ is a genus three curve and $\psi$ is a tetragonal pencil which does not contain fibers of the form $2p+2q$.
The trigonal construction \cite{donagi}*{Theorem~2.9} establishes an isomorphism $\cR_4^\mathrm{trig} \isoto \cm_3^\mathrm{tetr}$. See Section~\ref{sec:trigonal} for more details.

Recillas~\cite{Recillas1993} originally observed, by working with generic curves over the complex numbers, that the isomorphism above descends to a birational map $\cR_4 \xdashrightarrow{\;\,\sim\;\,} \ck_3$ where $\ck_3$ is the moduli space of Kummer varieties of dimension three. Hidalgo and Recillas~\cite{Hidalgo-Recillas1999} extend this to a map $\overline{\cR}_4 \to \ca_3$ and study the general fiber. This construction requires the identity section of $\ck_3$ to be blown-up, which corresponds to $\cR_4^\mathrm{odd}$ on the left hand side, giving another explanation for the disparate behavior of the odd locus.

Del Centina and Recillas~\cite{DelCentina-Recillas1983} study projective models of the double cover $\tilde C \to C$ of a generic curve $C$ of genus four. The generic member in $\cR_4^\mathrm{odd}$ and the case where $\Gamma_\varepsilon$ is a type~(2) cubic symmetroid is touched upon (see Definition~\ref{def:types}). Vakil~\cite{vakil--twelve} reviews the classical correspondences involving $\cR_4^\mathrm{odd}$.

Relations between curves of genus three and four have been observed more than a century ago, see for instance the articles by Roth~\cite{Roth1911} and Milne~\cite{Milne1923}. A remarkable treatment of this connection, with genus three curves and their Jacobians on the one hand and symmetric cubics containing genus four curves on the other is given in Coble's book~\cite{coble--theta-book}*{Sections 14 and 50}.

Prym varieties are generically Jacobians only when the base of the double cover is genus two, three or four. For the arithmetic aspects of the genus three case we refer to~\cite{Bruin2008} and to~\cite{ACGH:volI}*{Exercise~VI.F} for an overview. See~\cite{mumford--prym} for a comprehensive treatment of the general theory of Prym varieties, including a discussion of the hyperelliptic case.

\subsection*{Acknowledgments}
We would like to thank Bernd Sturmfels for constant feedback and for introducing us to one another during his ``tritangent summit'' where this project was conceived. We have also benefited from conversations with Gavril Farkas and Alessandro Verra which we gratefully acknowledge. Special thanks to Corey Harris for useful observations and computations at the start of this project. We would like to thank Igor Dolgachev for guidance into the literature and mathematical comments. We are also indebted to the anonymous referee for careful reading and detailed feedback, which resulted in significant improvements.

\section{Background and notation}

Let $k$ be a field of characteristic not two. When $V$ is a finite dimensional $k$-vector space and $X$ is a smooth projective $k$-variety with a line bundle $\cl$ we use the following notation:
\begin{itemize}
  \item $\omega_X$ : the canonical bundle of $X$,
  \item $|V|$ : the projective space of one dimensional subspaces of $V$,
  \item $\Pp V$ : the projective space of one dimensional quotients of $V$,
  \item $\H^0(X,\cf)$, $\H^0(\cf)$ : the global sections of the sheaf $\cf$ on $X$,
  \item $h^0(\cf)$ : the dimension of the $k$-vector space $\H^0(X,\cf)$,
  \item $Z(f)$ : for $f \in \sym^d V$ denotes the zero locus of $f$ in $\Pp V$,
  \item $\ksep,\, \kbar$ : a separable and algebraic closure of $k$ respectively,
  \item $X\sep$ : the pullback of $X/k$ to $\ksep$, that is, $X \times_{\spec k} \spec \ksep$,
  \item $\bs_{[\cl]}$ : the Brauer-Severi variety that provides a model of $|H^0(X\sep,\cl)|$ over $k$, when $[\cl]\in\Pic(X\sep)$ is a $\Gal(k\sep/k)$-invariant class (see~\cite{BruinFlynn2004} for a very explicit description).
  \item A conic is a curve of degree two (in space or in plane) and a quadric is a degree two surface in space.
\end{itemize}
 
\subsection{Genus 4 curves}\label{sec:prym_basic_defs}

The canonical model of a genus four non-hyperelliptic curve $C$ is the complete intersection of a quadric $Q_C$ and a cubic $\Gamma$ in $\PP^3$. There is a four-dimensional linear system of cubics containing $C$.

 A \emph{theta characteristic} on $C$ is a line bundle $\eta$ such that $\eta^{\otimes 2}\simeq \omega_C$. A theta characteristic is called \emph{odd} or \emph{even} according to the parity of $h^0(C,\eta)$. An even theta characteristic with $h^0(C,\eta)>0$ is called a \emph{vanishing theta-null}. 

  An odd theta characteristic $\eta$ corresponds to an effective divisor $D_\eta$ of degree three on $C$. Since $\eta^{\otimes 2}\simeq \omega_X$, the canonical model of $C$ has a \emph{tritangent} plane $H_\eta$ satisfying $H_\eta \cdot C = 2D_\eta$. We will often abuse notation and identify $\eta$ with $D_\eta$.

The rulings of the quadric $Q_C$ correspond to trigonal pencils on $C$. If $Q_C$ is nonsingular then the two rulings induce two distinct trigonal pencils on $C$. If $Q_C$ is singular then $C$ is uniquely trigonal. The latter happens precisely when $C$ admits a vanishing theta-null. A non-hyperelliptic genus four curve $C$ admits at most one vanishing theta-null. If $C$ has a vanishing theta-null $\theta$, then a two-torsion class $\varepsilon$ is called \emph{even} or \emph{odd} according to the parity of the corresponding theta characteristic $\theta\otimes \varepsilon$.

\begin{definition}\label{def:bielliptic} A class $\varepsilon \in \Jac_C[2]$ is called  \emph{bielliptic} if $C$ admits a double cover of a genus one curve $\xb\colon C\to E$ curve and a class $\varepsilon'\in\Jac_E[2]$ such that $\varepsilon=\xb^*(\varepsilon')$.
\end{definition}

\subsection{The trigonal construction} \label{sec:trigonal}
The trigonal construction~\cites{Recillas1974,donagi} establishes a relation between Prym varieties of trigonal curves and Jacobians of tetragonal curves, over algebraically closed base field. We recall the construction in order to formulate its implications when applied to curves over non-algebraically closed base fields.

The construction is based on a Galois-theoretic observation. Let $X\to L$ be a degree $4$ (ramified) cover of curves, where $L$ is of genus $0$. Then the Galois closure $X!$ of $X$ over $L$ generically is of degree $24$, with $\Aut(X!/L)=S_4$. The subgroups $C_4\subset D_4\subset S_4$ provide subcovers $\tC\to C\to L$. Recillas~\cite{Recillas1974} proves that $\Prym(\tC/C)=\Jac_{X}$ in the generic case and Donagi~\cite{donagi} generalizes this to stable limits (for special ramification types of $X\to L$).

From a Galois-theoretic perspective, $f_C\colon C\to L$ is characterized by being the \emph{cubic resolvent} of $f_X\colon X\to L$: for a point $p\in L$, the fiber $f_C^{-1}(p)$ parametrizes the ways in which $f_X^{-1}(p)$ can be split into two sets of two, counting multiplicity.

Conversely if we start with a cover $\tC$ of a trigonal curve $C\to L$, we find that the Galois closure of $\tC\to L$ generically is of degree $48$, with $\Aut(\tC!/L)=(C_2)^3\rtimes S_3\simeq C_2\times S_4$.
If $\tC\to C$ is unramified, then $\tC!$ is not geometrically connected, with the center of $\Aut(\tC!/L)$ interchanging the components. If we take one component $X!$ we obtain a Galois cover $X!\to L$ with $\Aut(X!/L)=S_4$, leading to a tetragonal cover $X\to L$ with $\Prym(\tC/C)=\Jac_X$.

This result generalizes to non-algebraically closed base field with a minor modification.

\begin{definition}
Let $\pi\colon\tC\to C$ be a separable degree $2$ map of proper nonsingular curves over a base field $k$. We say $\pi'\colon\tC'\to C$ is a \emph{quadratic twist} of $\pi$ if there is a quadratic extension $k(\sqrt{d})/k$ and an isomorphism $\psi\colon\tC\to \tC'$ over $k(\sqrt{d})$ such that $\pi=\pi'\circ\psi$.
\end{definition}

\begin{remark}
On the level of function fields we see that if $k(\tC)=k(C)[\sqrt{f}]$ then $k(\tC')=k(C)[\sqrt{df}]$. In fact, Kummer theory yields that the quadratic twists of $\pi$ are classified by $k^\times/k^{\times2}$.
\end{remark}

If $\tC\to C\to L$ is defined over a non-algebraically closed base field then two components of $\tC!$ may be only defined over a quadratic extension $k(\sqrt{d})$ of $k$. By taking the appropriate twist of $\tC$ we can assure the components of $\tC!$ are defined over $k$, so we obtain the following.

\begin{proposition}\label{P:trigonal_construction} Let $C,\, L$ be proper nonsingular curves with $L$ of genus $0$ and $\pi\colon C\to L$ finite of degree $3$. Let $\varepsilon$ be a line bundle of order two on $C$. Then there is a double cover $\pi\colon \tC\to C$ such that $\pi^*(\varepsilon)=0$ and a degree $4$ cover $X\to L$ such that $\Prym(\tC/C)=\Jac_X$.
\end{proposition}

\subsection{A remark on the odd case} \label{sec:odd_del_pezzo}

In~\cite{vakil--twelve} the correspondence between $\cR_4^\mathrm{odd}$ and smooth plane quartics $X$ with a canonical pencil is worked out in detail. An even closer link with del~Pezzo surfaces can be made explicit, which may be useful for the construction of examples. Let $C$ be a genus four curve with a vanishing theta-null $\theta$. Then $C$ can be realized as a curve in the branch locus of the Bertini involution on a del~Pezzo surface $S$ of degree one. The exceptional lines on $S$ correspond two-to-one with the odd theta characteristics on $C$. Let $E_\eta$ be such a line. By blowing down this line we get a degree two del~Pezzo surface $S'$, together with a point $p_\eta$. It has an involution (the Geiser involution), branched along a genus three curve $X$. The quotient by this involution gives us $X\subset\PP^2$, together with a point $\bar p_\eta$. Generically, projection from $\bar p_\eta$ gives a canonical tetragonal pencil $\cl$ on $X$.
This is the curve $X$ that Proposition~\ref{P:trigonal_construction} associates $C$ with $\varepsilon=\eta\otimes\theta^\vee$ and $\cl$ is the corresponding tetragonal pencil. Starting with the pair $(X, \cl)$, the construction can be reversed to obtain $(C,\varepsilon)$.

\section{Cubic symmetroids}\label{sec:cubic_symmetroids}

A \emph{symmetroid surface} in $\PP^3$ of degree $d$ is a surface where the defining equation is given as the determinant of a $d\times d$ symmetric matrix with entries that are linear forms in the coordinates on $\PP^3$. We are interested in characterizing the cubic surfaces that admit a symmetroid model. There is a large classical literature on the subject, see Dolgachev's book~\cite{dolgachev}*{\S 9.3.3} and the survey by Beauville~\cite{beauville}. In this section we review these results and extend the characterization to apply to non-algebraically closed base field as well.

\subsection{Basic properties}

For the rest of this section let $U$ be a three dimensional vector space and $V$ be a four dimensional vector space. We give coordinate-free constructions, but for expositional clarity we fix bases $U=\langle z_0,z_1,z_2 \rangle$ and $V = \langle x_0,x_1,x_2,x_3 \rangle $. In subsequent sections where $(C,\varepsilon)$ is in play, we have $V=\H^0(C,\omega_C)$ and $U=\H^0(C,\omega_C\otimes \varepsilon)$. 

Let $\ca\colon k \toi V\otimes \sym^2 U$ be given, which we equate to an element in $V \otimes \sym^2 U$ up to scaling. We consider $\sym^2 U$ as a subspace of $U \otimes U$ and obtain from $\ca$ two tensor contraction maps: $\ca_U\colon U^\vee \to V \otimes U ,\, \ca_V\colon V^\vee \to \sym^2 U$.  Notice that the order in which we contract $U$ does not matter due to the symmetry of $\ca$. We call $\ca$ \emph{non-degenerate} if $\ca_V$ is injective.

In coordinates, we can write the 3-tensor $\ca$ as a sum $\sum_{0 \le i,j \le 3 } l_{ij}\otimes z_iz_j$, where $l_{ij} \in V$ and $l_{ij} = l_{ji}$ for all $i,j$. Identifying quadratic forms with symmetric matrices we get:
\[
  \ca_V= \begin{pmatrix} l_{11} & l_{12} & l_{13} \\ l_{21} & l_{22} & l_{23} \\ l_{31} & l_{32} & l_{33} \end{pmatrix}.
\]
Alternatively, with $x_0^*,\ldots,x_3^*$ the basis for $V^\vee$ dual to $x_0,\ldots,x_3$ and $q_i=\ca_V(x_i^*)$, we have
$\ca_V=x_0q_0(z_0,z_1,z_2)+\cdots+x_3q_3(z_0,z_1,z_2)$, where $q_0,q_1,q_2,q_3$ are quadratic forms on $U^\vee$.  

There is a locus  $\cd \subset |\sym^2 U| \simeq \Pp^5$ parametrizing singular symmetric matrices which is called the \emph{discriminant hypersurface}. Naturally, $\cd$ is cut out by the determinant. Given $\ca \in V \otimes \sym^2 U$ we can use the linear map $\ca_V\colon V^\vee \to \sym^2 U$ to pull back the discriminant hypersurface to a degree three surface $\Gamma_\ca$ in $\Pp(V)$ which is cut out by the cubic equation $\det \ca_V=0$.

\begin{definition}
  Let $\Gamma \subset \Pp(V)$ be a cubic surface. A \emph{symmetrization} of $\Gamma$ is a non-degenerate tensor $\ca \in V \otimes \sym^2 U$ such that $\Gamma = \Gamma_\ca$.
\end{definition}

  One can view cubic symmetroids as cross-sections of the cubic fourfold $\cd \subset |\sym^2 U| \simeq \Pp^5$. Note that $\cd$ is singular along non-reduced conics, that is, along the image of the second Veronese of $\vpp = |U|$. The cross-sections will acquire singularities as they pass through these points. The generic cubic symmetroid, corresponding to the generic cross-section, has four simple nodes. If $\ca_V\colon V^\vee \to \sym^2 U$ is degenerate and has a one-dimensional kernel then $\Gamma_\ca$ is a cone over a plane cubic curve. This case is considered in Section~\ref{sec:degenerate_cubics}. 

\subsection{Classification of cubic symmetroids over algebraically closed fields}

\begin{theorem}[{\cite{catanese83},~\cite{dolgachev}*{\S 9.3.3}}]\label{T:symmetroid_classification}
  Let $\Gamma$ be a non-degenerate cubic symmetroid over an algebraically closed field of characteristic not equal to two. Then, up to projective transformation, it is of the following form:
\begin{enumerate}
    \item Cayley cubic; four nodes in general position: $ \det \begin{psmallmatrix} x_0 & x_3 & x_3 \\ x_3 & x_1 & x_3 \\ x_3 & x_3 & x_2  \end{psmallmatrix} $.
    \item Two nodes and one regular double point of type $A_3$:  $ \det \begin{psmallmatrix} x_0 & x_3 & -x_3 \\ x_3 & x_1 & 0 \\ -x_3 & 0 & x_2 \end{psmallmatrix} $.
    \item One node and one regular double point of type $A_5$: $ \det \begin{psmallmatrix} x_0 & x_2 & 0 \\ x_2 & x_1 & x_3 \\0 & x_3 & -x_2 \end{psmallmatrix} $.
    \item Non-normal, with singularity along $\{x_3=x_2=0\}$: $ \det \begin{psmallmatrix} x_0 & -x_3 & x_2 \\ -x_3 & x_1 & x_3 \\  x_2 & x_3 & 0 \end{psmallmatrix} $.
    \item Non-normal, with singularity along $\{x_0=x_2=0\}$:  $ \det \begin{psmallmatrix} 0 & x_2 & x_0 \\ x_2 & -x_0 & x_3 \\ x_0 & x_3 & x_1 \end{psmallmatrix} $.
    \item Reducible; union of a smooth quadric with a \emph{tangent} plane:  $ \det \begin{psmallmatrix} 0 & x_1 & x_3 \\ x_1 & 0 & x_2 \\ x_3 & x_2 & x_0 \end{psmallmatrix} $.
    \item Reducible; union of an irreducible quadric cone with a \emph{non-tangent} plane:  $ \det \begin{psmallmatrix} x_0 & 0 & 0 \\ 0 & x_1 & x_3 \\ 0 & x_3 & x_2 \end{psmallmatrix} $.
    \item Reducible and non-reduced; union of a double plane with another plane:  $ \det \begin{psmallmatrix} 0 & 0 & x_2 \\ 0 & x_0 & x_3 \\ x_2 & x_3 & x_1 \end{psmallmatrix} $.
  \end{enumerate}
\end{theorem}

\begin{definition}\label{def:types}
  Let $\Gamma$ be a cubic symmetroid over a field $k$ and $n$ an integer from one through eight. We will say $\Gamma$ \emph{is of type $(n)$} if the base change of $\Gamma$ to an algebraic closure of $k$ is of type $(n)$ according to the classification in Theorem~\ref{T:symmetroid_classification}.
\end{definition}

\subsection{Adjugation map}\label{sec:adjugation_map}

A symmetrization $\ca$ gives rise to a linear system of quadratic forms $\langle q_0,q_1,q_2,q_3 \rangle$ on $\pp = \Pp(U)$. If that system is sufficiently general, one expects that for a point $p\in\PP^2$ there is a unique conic in $\langle q_0,q_1,q_2,q_3 \rangle$ for which the singular locus is $p$. Let $u\in U^\vee$ be a representative of $p$. The singular locus of a conic is given by the kernel of its symmetric matrix, so $v\in V^\vee$ represents a conic with a singularity at $p$ precisely if $\ca_V(v)\cdot u=0$. 

In coordinates, we may compute the $1\times 4$ adjugate matrix $\adj \ca_U$ of the $3\times 4$ matrix $\ca_U$. The entries of $\adj \ca_U$ are cubic forms, inducing the desired map $\xc_{\ca}\colon \pp \ratto \ppp$. As the adjugate of a non-square matrix annihilates that matrix, we see that $\xc_\ca(p)$ represents a singular conic with a singularity at $p$. In particular, the image of $\xc_\ca$ lies in $\Gamma_\ca$. If $\xc_\ca$ is birational onto $\Gamma_\ca$, then for a generic point $p \in \pp$ there is a unique point $[v] \in \Gamma_\ca$ such that $\ca \cdot v$ is singular at $p$.

We obtain the following from a case by case analysis.

\begin{proposition}\label{prop:parametrization}
  The adjugation map $\xc_\ca\colon \Pp(U) \ratto \Pp(W)$ is birational onto $\Gamma_\ca$ if and only if the symmetrization $\ca$ is of type~(1),(2),(3),(4) or (5). \qed
\end{proposition}

\begin{remark}\label{rem:parametrization}
  If $\ca$ is of type~(6), then $\xc_\ca$ is a birational map onto the degree two component of $\Gamma_\ca$. If $\ca$ is of type~(7), then the image of $\xc_\ca$ is the one dimensional singular locus of $\Gamma_\ca$, which is a conic. If $\ca$ is of type~(8) then $\xc_\ca$ maps $\pp$ on to a conic contained in the non-reduced component.
\end{remark}

A coordinate-free description of the map $\xc_\ca$ can be obtained by considering the adjugation map
  \[
    V \otimes U = \hom(V^\vee, U) \overset{\adj}{\too} \hom(\bigwedge^3 V^\vee , \bigwedge^3 U) = V^\vee \otimes \det(V) \otimes \det(U).
  \]
  Composing $\ca_U\colon U^\vee \to V\otimes U$ with this adjugation map yields $\adj(\ca_U)\colon U^\vee \to V^\vee \otimes \det(V) \otimes \det(U)$.  After projectivization, we get $\xc_\ca\colon \Pp(U) \ratto \Pp(V)$.

\subsection{Gauss map and symmetrization}\label{sec:gauss_map}
The Gauss map of a reduced surface $S=Z(f) \subset \ppp$ is the map taking each smooth point $s$ of $S$ to the unique tangent plane $T_s S$ of $S$. In coordinates, the map $\gamma_S\colon \ppp \ratto \vppp:\; p \mapsto \nabla f(p)$ restricts on $S$ to the Gauss map of $S$. The closure of the image $\gamma_S(S)$ is the dual variety $\widehat{S}\subset\widehat{\PP}^3$.

The Gauss map of a cubic symmetroid $\Gamma_\ca$ can be expressed in terms of the symmetrization $\ca$. Consider $\ca_V\colon V^\vee \to \sym^2 U$ as defining a map  $\pp \ratto \vppp$, such that a linear form $v\in V^\vee$ on the codomain is pulled back to the conic $\ca_V(v)$. In bases, if $x_0^*,\dots,x_3^*$ is a basis for $V^\vee$ and $q_i := \ca_V(x_i^*)$ then the Gauss map of $\Gamma_\ca$ is given by $\xq_\ca\colon  \pp \ratto \vppp;\; [z] \mapsto [q_0(z),\dots,q_3(z)]$. 
\begin{remark}\label{rem:q_and_A} 
  The map $\xq_\ca$ gives rise to a $3$-tensor that is a multiple of the tensor $\ca$.
\end{remark}

\begin{remark}\label{rem:regular-xq}
  If $\ca$ is of type~(1), (2), (3) or (7) then the rational map $\xq_\ca$ is in fact a \emph{regular} map. This can be seen by direct computation. It is fortunate that only these four types of non-degenerate cubic symmetroids will arise in the Prym construction in Section~\ref{sec:prym_construction}.
\end{remark}

\begin{lemma}\label{lem:obvious}
For any point $p=[v] \in \ppp$, the conic cut out by $\ca \cdot v$ in $\pp$ equals the pullback of the dual plane $H_p \subset \vppp$ of $p$ via $\xq_\ca\colon \pp \to \vppp$.
\end{lemma}

\begin{proof}
  The representative $v$ of $p$ is in $V^\vee$ and the dual plane $H_p$ is the zero set of $v$ in $\vppp$. Pulling back $H_p$ is equivalent to computing the zero locus of the pullback form $v$, which is $\ca_V(v) = \ca \cdot v$ by definition of $\xq_\ca$.
\end{proof}

The following well known result links the symmetrization $\ca$ of $\Gamma_\ca$ to the Gauss map $\gamma_\ca\colon\Gamma_\ca \ratto \vppp$, see also~\cite{dolgachev}*{Remark 4.1.2}. 

\begin{lemma}\label{lem:gauss}
  If $\ca$ is of type~(1)--(6), then the rational maps $\xq_\ca,\gamma_\ca,\xc_\ca$ satisfy $\xq_\ca = \gamma_\ca \circ \xc_\ca$.
\end{lemma}
\begin{proof}
One can check this via a case-by-case analysis. We  give a more insightful proof.

Let $p=[v] \in \Gamma_\ca$ be a smooth point. By Lemma~\ref{lem:obvious} the pullback of the dual plane $H_p$ to $\pp$ is the conic $Z(\ca\cdot v)$ cut out by $\ca \cdot v$. As the point $p$ belongs to $\Gamma_\ca$, the conic $Z(\ca\cdot v)$ is singular. Since $p$ is a smooth point of $\Gamma_\ca$, the conic $Z(\ca \cdot v)$ is the union of two distinct lines coming together at a point $r \in \pp$. If $\ca$ is of type~(1)--(5) then apply Proposition~\ref{prop:parametrization} together with the preceding explanation of $\xc_\ca$ to conclude $\xc_\ca(r)=p$. If $\ca$ is of type~(6) then recall that the linear component of $\Gamma_\ca$ is tangential to the quadric component of $\Gamma_\ca$ and apply Remark~\ref{rem:parametrization}.

Since $\xq_\ca^*(H_p)$ is singular at $r$, the plane $H_p$ must be tangential to the dual surface $\vGam_\ca$ at $\xq_\ca(r)$. Using the reflexivity theorem~\cite{tevelev}*{Chapter 1} between dual hypersurfaces, we conclude that $\xq_\ca(r)=\gamma_\ca(p)=\gamma_\ca \circ \xc_\ca(r)$. 
\end{proof}

\subsection{Classification of cubic symmetroids over non-algebraically closed fields}\label{sec:cubic_classification_over_k}

From Theorem~\ref{T:symmetroid_classification} it follows that if a cubic surface $\Gamma$ admits a symmetrization then the type of the symmetrization is determined by the singular locus of $\Gamma$. Let $\Gamma$ be a cubic surface over a field $k$. In this section we prove that if $\Gamma$ admits a symmetrization over an algebraic closure of $k$, then it admits one over $k$. First we consider the general case, that is $\Gamma$ of type~(1).

\begin{definition}
  Let $k$ be a field of characteristic different from $2$. A cubic surface $\Gamma$ over $k$ is a \emph{Cayley cubic} if its singular locus is of dimension $0$, of degree $4$, is reduced, separable, and is not contained in a plane.
\end{definition}

It follows from the definition of a Cayley cubic that the singular locus is $\spec(R)$ for some quartic \'etale algebra $R$ over $k$. Then $\Gamma$ has a singular point $p_R$ defined over $R$. Furthermore, the complement of $p_R$ in the singular locus is contained in a plane, given by a linear form $h$ over $R$.
Without loss of generality, we can assume that $R=k[\alpha]$ and write
\[h=h_0(\alpha)x_0+h_1(\alpha)x_1+h_2(\alpha)x_2+h_3(\alpha)x_3.\]

We consider the Norm and Trace maps for the finite extension $k[x_0,x_1,x_2,x_3][\alpha]/k[x_0,x_1,x_2,x_3]$, which restrict to the ordinary Norm and Trace maps of $R/k$.

\begin{lemma}
	Let $\Gamma$ be a Cayley cubic over the field $k$ with singular locus given by coordinate ring $R$. Define $p_R$ and $h$ as above. Then for some $\delta\in R$ we have
	\[
	\Gamma\colon\Trace\left(\delta
	\frac{\Norm(h)}{h}
	\right)=0.\]
	In fact, by scaling $h$ we can assume $\delta=1$.
\end{lemma}
\begin{proof}
	Over an algebraic closure we know $\Gamma$ is projectively equivalent to
	\[x_0x_1x_2+x_0x_1x_3+x_0x_2x_3+x_1x_2x_3=0.\]
	In this case, take $p_R=(1:0:0:0)$ and $h=x_0$. With a change in coordinates we can explicitly represent the singular locus as an affine variety to compute $\Norm(h)=x_0x_1x_2x_3$, and the statement of the theorem holds. The general statement follows by covariance under projective equivalence, together with the basic observation that the given equation is defined over the base field. For the second assertion, set $h'=\delta\inv h$ and note
	\[
    \Trace\left(\delta\frac{\Norm(h)}{h}\right)=\Norm(\delta)\Trace\left(\frac{\Norm(h')}{h'}\right).\qedhere
	\]
\end{proof}

\begin{lemma}\label{L:cayley-k-equivalence}
	Two Cayley cubics $\Gamma_1,\, \Gamma_2$ over $k$ are projectively equivalent if and only if their singular loci are isomorphic over $k$.
\end{lemma}
\begin{proof}
  One direction is obvious. Suppose now that the singular loci of $\Gamma_1,\,\Gamma_2$ are isomorphic. Since the singular loci are of the same degree, these loci must be projectively equivalent. After a linear transformation we assume that the singular loci of $\Gamma_1$ and $\Gamma_2$ coincide. Then, for a fixed form $h$ we may write the defining equations as
	\[
	\Gamma_i\colon \Trace\left(\delta_i
	\frac{\Norm(h)}{h}
	\right)=0,\]
If $A$ is a linear transformation over $k$ that leaves $h$ invariant, we must have that $ h\circ A=\delta_A h$ for some $\delta_A \in R^*$. In fact, we have an isomorphism between $R^*$ and the group of such transformations. By finding $A$ for which $\delta_A=\delta_2/\delta_1$ we see that $\Gamma_2\circ A$ and $\Gamma_1$ are both cut out by the same equation.
\end{proof}

\begin{proposition}\label{P:cayley_is_symmetroid}
	A Cayley cubic $\Gamma$ defined over $k$ admits a symmetrization over $k$. 
\end{proposition}

\begin{proof} Let $R$ be the quartic \'etale algebra of the singular locus of $\Gamma$.  Pick a primitive element $\theta\in R$ and express the minimal polynomial of $\theta$ as $a_0+a_1t+a_2t^2+a_3t^3+t^4$ for $a_0,\dots,a_3 \in k$.
	
Consider the cubic symmetric threefold in $\PP^4$ defined by
\[\cd'\colon \det\begin{pmatrix}u_0&u_1&u_2\\u_1&u_2&u_3\\u_2&u_3&u_4\end{pmatrix}=0.\]
Its singular locus $\cd_1'$ consists of the locus where the symmetric matrix above has rank $1$. Therefore, $\cd_1'$ is the rational normal curve parametrized by $u_i = t^{i-1}$.
Intersect $\cd'$ with the hyperplane $A= \{a_0u_0+a_1u_1+a_2u_2+a_3u_3+u_4=0\}$ and obtain the cubic symmetroid $\Gamma'=\cd'\cap A$. Observe that the locus $\Gamma'\cap \cd'_1=\spec(R')$ is singular, where $R'=k[t]/(a_0+a_1t+a_2t^2+a_3t^3+t^4)$ is identified with $R$. A priori, the surface $\Gamma'$ could acquire other singularities if $A$ is tangent to $\cd'$ in some nonsingular point. However, direct computation of the dual variety $\widehat{\cd}'$ of $\cd'$ shows that $A$ is tangent to $\Gamma'$ away from the singular locus only if $\disc_t(a_0+a_1t+a_2t^2+a_3t^3+t^4)=0$, which is not the case, because this is the minimal polynomial of a generator of an \'etale algebra.

Consequently, the singular loci of $\Gamma'$ and $\Gamma$ are identified. By Lemma~\ref{L:cayley-k-equivalence}, $\Gamma$ and $\Gamma'$ are projectively equivalent over $k$. Therefore, $\Gamma$ admits a symmetrization over $k$.
\end{proof}

\begin{proof}[Proof of Proposition~\ref{prop:intro_k_symmetrization}]
We split the proof by type according to the classification in Theorem~\ref{T:symmetroid_classification}.
Proposition~\ref{P:cayley_is_symmetroid} establishes the theorem if $\Gamma$ is of type (1). The other possibilities for irreducible $\Gamma$ are (2)--(5). For each of those, there is a rational singular point and one can check that projection from it gives a birational map to $\PP^2$. In each case on can check that, up to some Cremona transformations determined by the singularities, this map is a birational inverse to the adjugation map $\xc_\ca$ from Proposition~\ref{prop:parametrization}. This shows that $\xc_\ca$ is defined over $k$. From Lemma~\ref{lem:gauss} it follows that the map $\xq_\ca$ is defined over $k$ and by Remark~\ref{rem:q_and_A} we can recover $\ca$.

Types (6)--(8) are each the union of a plane and a (possibly degenerate) quadric, which are necessarily each defined over $k$. The quadric is the determinant locus of a $2\times 2$ symmetric matrix and the plane is the determinant locus of a $1\times 1$ symmetric matrix. The union is the determinant locus of the $3\times 3$ determinant locus of the block diagonal symmetric matrix obtained from the two.
\end{proof}

\begin{remark} Many of the types can be obtained from the same construction in the proof of Proposition~\ref{P:cayley_is_symmetroid} by taking non-\'etale algebras $R=k[t]/f(t)$, where $f(t)$ is some quartic polynomial. One can check directly that the different factorization types of $f(t)$ lead to irreducible cubics:
	\begin{center}
		\begin{tabular}{l|l}
			Type&ramification type of $f(t)$\\
			\hline
			(1)& $\displaystyle t^4+a_1t^3+a_2t^2+a_3t+a_4$\\
			(2)& $\displaystyle(t-\alpha_1)^2(t^2+a_1t+a_2)$\\
			(3)& $\displaystyle(t-\alpha_1)^3(t-\alpha_2)$\\
			(4)& $\displaystyle(t^2+a_1t+a_2)^2$\\
			(5)& $\displaystyle(t-\alpha_1)^4$
		\end{tabular}
	\end{center}
\end{remark}

\subsection{A natural double cover} \label{sec:natural_double_cover}

Let $\cd \subset |\sym^2 U|$ be the discriminant locus of quadratic forms on $\PP(U)$. 
An element $[A]\in \cd$ defines a singular quadratic form $q_A$ on $\PP(U)=|U^\vee|$. 
It is a standard result on determinantal varieties that the singular locus $\cd_1$ of $\cd$ corresponds to quadrics $A$ with $\rk(A)=1$.

A singular quadratic form on $\PP^2$ describes a pair of lines (coinciding if $\rk(A)=1$).
By labeling these lines, we obtain an unramified double cover $\tilde{\cd}\to \cd-\cd_1$.

If we represent our quadratic form $A$ with a symmetric matrix $(a_{ij})_{i,j =1,\dots,3}$, $a_{ij}=a_{ji}$, then we get a relation $(a_{11}a_{22}-a_{12}^2)(a_{22}a_{33}-a_{23}^2)-(a_{12}a_{23} - a_{13}a_{22})^2=a_{22}\det(A)$.
Consequently, up to birationality, the double cover $\tilde{\cd}$ can be obtained by adjoining any of the roots $\sqrt{a_{12}^2-a_{11}a_{22}}$, $\sqrt{a_{13}^2-a_{11}a_{33}}$, $\sqrt{a_{23}^2-a_{22}a_{33}}$ to the function field of $\Gamma$. The vanishing of these three minors on $\cd$ coincides with $\cd_1$.

In order to extend the double cover we consider $\cs=\{([A],[u])\in \cd\times\PP(U) \mid Au=0\}$. This space parametrizes singular conics on $\PP(U)$ together with a point $u$ in the singular locus. For $[A]\in \cd-\cd_1$ we have a unique singularity, so on $\cd-\cd_1$ the map $\cs\to\cd$ is an isomorphism.

  Let $\ct \subset \Pp(U) \times |U|$ be the tautological $\p$-bundle over $\Pp(U)$: for each point $x \in \Pp(U)$ the fiber $\ct_x \simeq \p$ parametrizes the lines through $x$. We define $\tilde \cs \subset \cd \times \ct$ as a double cover of $\cs \subset \cd \times \Pp(U)$ as follows:
$\tilde \cs := \{(A,x,y) \in \cd \times \ct \mid A x = 0,\,  \bar q_A(y)=0\}$, where $\bar q_A$ is the conic induced on $\ct_x$ by $A$. Roughly speaking, $\tilde \cs$ parametrizes singular conics $q$ in $\Pp(U)$ together with a point of singularity $x$ as well as one of the two branches of $q$ at $x$. 

\begin{remark}
The fiber of $\cs \to \cd$ over $A\in\cd$ of rank 1 is a line, parametrizing all points of the double line $q_A$. The double cover $\tilde \cs \to \cs$ then is branched all along this line. More generally, the branch locus of $\tilde \cs \to \cs$ is precisely the pullback of the singular locus in $\cd$.
\end{remark}

A symmetrization $\ca$ of a cubic surface $\Gamma$ gives an embedding $\Gamma\subset \cd$.
We write $S_\ca$ for the pullback of $\Gamma_\ca$ along $\cs\to\cd$. The surface $S_\ca$ forms a partial resolution of the singularities of $\Gamma_\ca$. Restriction of $\tilde{\cs}\to\cs$ yields a double cover $\tilde{S}_\ca\to S_\ca$.

By restricting the cover $\tilde{\cs}\to\cd$ to $\Gamma_\ca$, we obtain a cover $\tilde{S}_\ca\to\Gamma_\ca$, which is finite unramified of degree $2$ outside $\cd_1\cap\Gamma_\ca$ (which is the singular locus of $\Gamma_\ca$).

\begin{remark}\label{rem:cone_double_cover}
	If $\Gamma$ is a cone over a smooth plane cubic curve, then this construction gives an unramified double cover of the cubic curve. This double cover gives rise to a double cover of $\Gamma$ branched only at its vertex.
\end{remark}

\begin{theorem}[Catanese~\cite{catanese81}]\label{thm:unique_double_cover}
  Over an algebraically closed base field, any cubic symmetroid of type (1), (2), or (3) has a unique double cover branched only over the singularities. 
\end{theorem}

\begin{corollary}\label{cor:unique_moduli}
  Any double cover of an irreducible symmetroid of type (1), (2), or (3) branched only over the singular points is a twist of the cover constructed here.
\end{corollary}

\begin{remark}\label{rem:unique_symmetrization}
  In light of this observation, Lemma~\ref{lem:gauss} implies that the symmetrization of $\Gamma_\ca$ is unique over $\kbar$, up to equivalence, if $\ca$ is of type (1)--(3). The inverse of the map $\xc_\ca$ is obtained by blowing up each of the singular points of $\Gamma_\ca$ exactly once and then blowing down the $(-1)$-curves. Using the Gauss map $\gamma_\ca$ we obtain $\xq_\ca=\gamma_\ca \circ \xc_\ca$.
\end{remark}

\subsection{Degenerate cubic symmetroids}\label{sec:degenerate_cubics}

Suppose now that $\ca_V\colon V^\vee \to \sym^2 U$ has a one-dimensional kernel $K \subset V^\vee$. Let $\ann(K) \subset V$ be the annihilator of the subspace $K$. The injection $V^\vee/K \toi \sym^2 U$ induced by $\ca$ allows us to pull back the discriminant locus $\cd \subset |\sym^2 U|$ onto a cubic plane curve $E \subset |V^\vee/K| = \Pp \ann(K)$. The corresponding cubic symmetroid $\Gamma_\ca$ is the cone over the cubic curve $E$ via the projection $\Pp V \ratto \Pp \ann(K)$. 

\begin{definition}
  For the rest of this paper, a symmetrization $\ca$ is \emph{degenerate} if $\ca_V$ has a one-dimensional kernel, and the cubic curve $E$ constructed above is \emph{smooth}. The cubic symmetroid $\Gamma$ is called a \emph{degenerate} cubic symmetroid if $\ca$ is degenerate.
\end{definition}

The study of a degenerate cubic symmetroid $\Gamma$ reduces to the study of the cubic curve $E \subset \Pp \ann(K)$ obtained by projecting $\Gamma$ from its node $[\ann(K)] \in \Pp V$. We refer to~\cite{dolgachev}*{Chapters 3--4} for the study of cubic plane curves and their symmetrizations. Let us recall the following two results from \emph{loc.\ cit.}

\begin{proposition}\label{prop:sym_of_E}
  Symmetrizations of a smooth plane cubic $E$ are canonically in bijection with the three elements of order two in $\jac_E[2]$.
\end{proposition}
\begin{remark}
  Since $E$ is smooth, the singular conics parametrized by $E$ are all reduced.
\end{remark}

An element of order two $\varepsilon \in \jac_E[2]$ induces an unramified double cover $\tilde E_\varepsilon \to E$, where $\tilde E_\varepsilon = \Spec_E (\co_E \oplus \varepsilon)$ with $\co_E \oplus \varepsilon$ having the natural algebra structure. In addition, the corresponding symmetrization induces a natural unramified double cover of $E$ by marking the two components of each of the singular conics parametrized by $E$. From~\cite{dolgachev}*{Exercise~3.2} we get:

\begin{proposition}
  The two double covers of $E$ described above coincide.\qed
\end{proposition}

\section{The Prym-canonical map}\label{sec:prym_canonical_map}

Throughout the section, $C$ is a non-hyperelliptic smooth proper curve of genus four and $\varepsilon$ a line bundle on $C$ of order two. We assume $C$ is embedded into $\ppp$ via its canonical map and denote by $Q_C$ the unique quadric containing $C$. We remove all generality assumptions from Catanese~\cite{catanese83}*{Theorem 1.5} and prove the following theorem. The proof of this theorem is divided into subsections corresponding to different types of $\varepsilon$.

\begin{theorem}\label{thm:prym}
  The line bundle $\varepsilon$ corresponds naturally to a pair $(\Gamma_\varepsilon,\xq_\varepsilon)$ where $\Gamma_\varepsilon \subset \ppp$ is a cubic symmetroid containing $C$ and $\xq_\varepsilon\colon \pp \to \vppp$ is a symmetrization of $\Gamma_\varepsilon$. 
\end{theorem}

This correspondence admits different interpretations with suitable restrictions on $\varepsilon$. We list the most prominent ones here.

\begin{theorem}\label{thm:induced_double_cover}
  If $(C,\varepsilon)$ is not odd, the natural double cover of $(\Gamma_\varepsilon,\xq_\varepsilon)$ restricts on $C$ to an unramified double cover which coincides with the double cover $\tilde C_\varepsilon \to C$ associated to $\varepsilon$.
\end{theorem}

\begin{notation}\label{not:prym_canonical}
  The line bundle $\omega_C \otimes \varepsilon$ induces the \emph{Prym-canonical map} $\rho_\varepsilon\colon C \to \pp_\varepsilon$ where $\pp_\varepsilon := \Pp \H^0(\omega_C\otimes \varepsilon)$. We will use $C_\varepsilon$ to refer to the image $\rho_\varepsilon(C)$.
\end{notation}

\begin{theorem}
  If $\varepsilon$ is not odd or bielliptic, then the adjunction map $\xc_\varepsilon\colon \pp_\varepsilon \ratto \ppp$ of $C_\varepsilon$ is defined by cubics, the image of $\xc_\varepsilon$ is $\Gamma_\varepsilon$ and the adjugate of $\xc_\varepsilon$ is the symmetrization $\xq_\varepsilon$.
\end{theorem}

The following theorem highlights our insistence on keeping track of symmetrizations as well as the cubic symmetroids.

\begin{theorem}
  If $\varepsilon$ is bielliptic then $\Gamma_\varepsilon$ is a cone over a cubic plane curve $E \subset \pp$. The symmetrizations of $\Gamma_\varepsilon$ correspond to symmetrizations of $E$, each inducing a different line bundle of order two on $C$ (one of which is $\varepsilon$).
\end{theorem}

Over the algebraic closure of the base field, the symmetroids of type (1)--(3) admit a unique symmetrization (Remark~\ref{rem:unique_symmetrization}) whereas a degenerate cubic symmetroid is a cone over a plane cubic and admits three symmetrizations (Section~\ref{sec:degenerate_cubics}).  

\begin{corollary}
  If $Q_C$ is smooth, then isomorphism classes of line bundles of order two on $C$ are in bijection with irreducible cubic symmetroids containing $C$ taken together with their symmetrizations.
\end{corollary}

If $Q_C$ is singular, then there are infinitely many cubic symmetroids of type (7) containing $C$---consider the union of $Q_C$ with a plane that is not tangent to $Q_C$. Only $120$ of these symmetroids appear via the construction in Theorem~\ref{thm:prym} and their linear component is a tritangent of $C$. The odd case is different in many ways, see Section~\ref{sec:odd_case}.

\begin{corollary}\label{cor:odd_is_evil}
  The sextic $C$ is the complete intersection of $\Gamma_\varepsilon$ and $Q_C$ if and only if $\Gamma_\varepsilon$ is not of type~(7). The symmetroid $\Gamma_\varepsilon$ is of type~(7) precisely when $\varepsilon$ is odd. 
\end{corollary}

In proving these results we make sure that the relevant constructions can be performed over the base field $k$. However, as the main properties of these constructions can be checked by passing to an algebraic closure of $k$ we will assume until Section~\ref{sec:descend_epsilon} that $k$ is algebraically closed to ease the exposition.
To ensure that our constructions carry through when $k$ is not algebraically closed, we prove the result below in Section~\ref{sec:descend_epsilon}.

\begin{proposition}\label{prop:descend_epsilon}
  If $[\varepsilon]\in \Pic(C^{\sep})[2]$ is $\Gal(\ksep/k)$-invariant then there exists a line bundle $\varepsilon$ on $C$ over $k$ that represents the class $[\varepsilon]$.
\end{proposition}

\subsection{The general case: irreducible symmetroids}\label{sec:prym_map_general}

Suppose that $(C,\varepsilon)$ is neither odd nor bielliptic. We use Notation~\ref{not:prym_canonical} freely.

\begin{theorem}\label{thm:prym_map_general}
  Adjunction on the Prym canonical curve $C_\varepsilon \subset \pp_\varepsilon$ gives a cubic rational map $\xc_\varepsilon \colon \pp_\varepsilon \ratto \ppp$ such that the image of $\xc_\varepsilon$ is an irreducible non-degenerate cubic symmetroid $\Gamma_\varepsilon$ containing $C \subset \ppp$. Furthermore, every irreducible non-degenerate cubic symmetroid containing $C$ arises in this way.
\end{theorem}

The method of proof follows that of Catanese~\cite{catanese83}. Our contribution to this case is Lemma~\ref{lem:mult_seq} which allows us to take Catanese's argument to its natural limit. This gives us the additional case where $(C,\varepsilon)$ is even.

\begin{proof}[Proof of Theorem~\ref{thm:prym_map_general}]
  Resolve the singularities of the Prym canonical sextic $\rho_\varepsilon(C) \subset \pp_\varepsilon$ by a minimal number of blow-ups $S_\varepsilon \to \pp_\varepsilon$. The linear system $\omega_S(C)$ induces a map $S_\varepsilon \to \ppp$ which restricts on $C \subset S_\varepsilon$ to the canonical embedding $C \toi \ppp$. The image of $S_\varepsilon$ in $\ppp$ is an irreducible non-degenerate cubic symmetroid (Lemma~\ref{lem:S_maps_to_Gamma}). The inverse of the blow-up map $\pp_\varepsilon \ratto S_\varepsilon$ composed with $S_\varepsilon \to \ppp$ gives the desired map $\xc_\varepsilon \colon \pp_\varepsilon \ratto \ppp$, which is defined by cubics. The last sentence of the theorem follows from Lemma~\ref{lem:prym_map_reverse}.
\end{proof}

\begin{definition}
   In resolving $C_\varepsilon$ one blow-up at a time the multiplicity of the node being blown-up is recorded at each step to form the \emph{multiplicity sequence of $C_\varepsilon$}. The sequence is ordered to be non-decreasing.
\end{definition}

\begin{lemma}\label{lem:mult_seq}
  The multiplicity sequence of $C_\varepsilon$ is $(3,2,2,2)$ when $\varepsilon$ is odd and $(2,2,2,2,2,2)$ otherwise.
\end{lemma}
\begin{proof}
  The computation given in~\cite{catanese83}*{pg.\ 36--37} implies that the multiplicity sequence of $C_\varepsilon$ is either $(3,2,2,2)$ or $(2,2,2,2,2,2)$. Now use Proposition~\ref{prop:triple_point}, which states that $C_\varepsilon$ admits a triple point if and only if $\varepsilon$ is odd.
\end{proof}

The curve $C$ embeds into $S_\varepsilon$ as the proper transform of $C_\varepsilon$. By adjunction, the line bundle $\omega_{S_\varepsilon}(C)$ restricts to the canonical bundle on $C$.

\begin{lemma}\label{lem:S_maps_to_Gamma}
The line bundle $\omega_{S_\varepsilon}(C)$ induces a map $\psi_\varepsilon\colon S_\varepsilon \to \ppp$ which restricts to the canonical embedding of $C$. The image of the map $\psi_\varepsilon\colon S_\varepsilon \to \ppp$ is a cubic symmetroid.
\end{lemma}
\begin{proof}
  The degree of $\psi_\varepsilon(S_\varepsilon)$ is readily computed as in~\cite{catanese83}*{pg.\ 37}. We replace the hypothesis in \emph{loc.\ cit.}\ that $C$ admits no vanishing theta-null with the observation in Lemma~\ref{lem:mult_seq} that the multiplicity sequence of $C_\varepsilon$ is $(2,2,2,2,2,2)$. That $\Gamma_\varepsilon$ is a cubic symmetroid follows from~\cite{catanese81}*{Theorem 2.19}  as demonstrated in~\cite{catanese83}*{Theorem 1.3}.
\end{proof}

Composing $\psi_\varepsilon$ with the rational inverse of the blow-up map $\pi_\varepsilon$ gives the rational map $\xc_\varepsilon\colon \pp_\varepsilon \ratto \ppp$ whose image is a cubic symmetroid containing $C$. 
A divisor class computation on $S_\varepsilon$ implies that $\xc_\varepsilon$ is defined by a linear system of cubics.

Conversely, take an irreducible non-degenerate cubic symmetroid $\Gamma$ containing $C \subset \ppp$ and a symmetrization $\xq\colon\pp \to \ppp$ of $\Gamma$. We have an induced double cover of $\Gamma$ branched over the nodes (Section~\ref{sec:natural_double_cover}) which induces a line bundle $\varepsilon$ on $C$ of order two. Let $S \to \Gamma$ be a minimal resolution of singularities and identify $C$ with its preimage in $S$.

\begin{lemma}\label{lem:prym_map_reverse}
  There is a map $\pi\colon S \to \pp$ whose restriction $\pi|_C\colon C \to \pp$ is $\rho_\varepsilon$.
\end{lemma}
\begin{proof}
  Let $H \subset S$ be the pullback of a hyperplane section from $\Gamma$. There is a divisor $L$ on $S$ such that $2L$ is equivalent to the reduction of the exceptional divisor of $S \to \Gamma$ (see proof of~\cite{catanese83}*{Theorem 1.3}). The line bundle $\co_S(H-L)$ induces the map $\pi$ restricting to $\rho_\varepsilon$ on $C$ as computed in~\cite{catanese83}*{end of p.~37}.
\end{proof}

\subsection{The bielliptic case}\label{sec:bielliptic} Let us recall that symmetrizations of a degenerate cubic $\Gamma$ are in bijection with the double covers of $\Gamma$ (Section~\ref{sec:degenerate_cubics}).

\begin{proposition}
  There is a natural bijection between bielliptic elements of order two $\varepsilon \in \jac_C[2]$ and symmetrizations of degenerate cubic symmetroids containing $C$.
\end{proposition}
\begin{proof}
  Assume $\varepsilon$ is bielliptic and let $\xb\colon C \to E$, $\varepsilon' \in \jac_E[2]$ such that $\xb^*\varepsilon'\simeq\varepsilon$ be as in Definition~\ref{def:bielliptic}. We prove in Lemma~\ref{lem:mu_exists} that there is a unique cubic line bundle $\mu$ on $E$ satisfying $\xb^*\mu \simeq \omega_C$. This provides a cubic model $E\hookrightarrow \PP^2=|\mu|^\vee$. 
  
  Let $f\colon  E \to \pp = \Pp \H^0(\mu)$ denote the cubic embedding of $E$ via the line bundle $\mu$.  The inclusion $\H^0(\mu) \toi \H^0(\omega_C)$ yields a projection map $\pr\colon \Pp \H^0(\omega_C) \ratto \Pp \H^0(\mu)$ satisfying $\pr|_C = f\circ \xb$. Therefore, $C$ is contained in the cone over $f(E) \subset \pp$.

Conversely, if $\Gamma$ is a cone over a plane cubic curve $E$ and $\Gamma$ contains $C$ then projecting from the node of $\Gamma$ recovers a double cover $C \to E$. Symmetrizations of $\Gamma$ and $E$ are in bijection. Now use Proposition~\ref{prop:sym_of_E}.
\end{proof}

\begin{lemma}\label{lem:mu_exists}
  Let $\xb\colon C \to E$ be a double cover of a curve of genus one. There is a degree three line bundle
  $\mu$ on $E$ over $k$, unique up to isomorphism, such that $\xb^*\mu \simeq \omega_C$.
\end{lemma}
\begin{proof}
  Use the Riemann--Hurwitz sequence to see that the ramification divisor $R \subset C$ of $\xb$ is a canonical divisor. Let $B = \xb(R)$ be the branch locus on $E$. By the theory of double covers, there is a line bundle $\mu$ of degree three such that $\mu^{\otimes 2} \simeq \co_E(B)$ and $\xb_*\co_C = \co_E \oplus \mu^\vee$. It is a standard observation that the natural map $\xb^* \mu^\vee \to \co_C$ realizes $\xb^*\mu^\vee$ as the ideal sheaf of the ramification locus. Consequently, $\xb^* \mu \simeq \co_C(R) \simeq \omega_C$. Uniqueness of such $\mu$ follows from the injectivity of the pullback map $\xb^*\colon  \jac_E \to \jac_C$~\cite{mumford--prym}*{pg.\ 332}.
\end{proof}

\begin{proposition}
  Using the notation above, the Prym canonical map $\rho_\varepsilon \colon  C \to \pp$ factors through $\xb\colon  C \to E$ and a map $f'\colon E \to \pp$ given by the complete linear system $\mu\otimes \varepsilon'$.
\end{proposition}
\begin{proof}
  Since $\xb^* \mu \simeq \omega_C$ and $\xb^*\varepsilon' \simeq \varepsilon$ we find $\xb^*(\mu\otimes\varepsilon') \simeq \omega_C \otimes \varepsilon$. Since both linear systems have the same number of sections, we are done.
\end{proof}

\subsection{The odd case}\label{sec:odd_case}

Assume $C$ admits a vanishing theta-null $\theta$, fix an odd theta characteristic $\eta$, and let $\varepsilon=\theta\otimes \eta^\vee$. Recall $C$ is contained in a singular quadric $Q_C$. We argue that the natural cubic symmetroid $\Gamma_\varepsilon$ associated with $(C,\varepsilon)$ is the union $Q_C \cup H_\eta$.

Together with Proposition~\ref{prop:triple_point}, the computation given in~\cite{catanese83}*{pg.\ 36--37} implies that the multiplicity sequence of $C_\varepsilon$ is $(3,2,2,2)$ and the image of the adjunction map $\xc_\varepsilon \colon \pp \ratto \ppp$ is $Q_C$. To further study $\xc_\varepsilon$ we state the following theorem, see Section~\ref{sec:odd_singularities} for the proof.

\begin{theorem}
  The curve $C_\varepsilon$ has a unique triple point $n_0$ and it satisfies $\rho_\varepsilon^*(n_0) \equiv \eta$. There is a bijection between points $u\in C$ with $u \le \eta$ and singularities $n_u$ of $C_\varepsilon \setminus n_0$. The bijection is so that $\langle n_0,n_u \rangle$ is the maximally tangent line to the branch of $C_\varepsilon$ at $n_0$ containing $u$. Furthermore, there exists a unique line $L_0$ satisfying $L_0 \cap C_\varepsilon = \{n_u \mid u \le \eta\}$. If $\ell u \le \eta$ but $(\ell+1)u \not\le \eta$ then $n_u$ is of type $A_{2\ell-1}$ or $A_{2\ell}$. The ray of $Q_C$ containing $u$ is tangential to $C$ iff $n_u$ is of type $A_{2\ell}$.
\end{theorem}

From this description, the geometry of the map $\xc_\varepsilon$ can be made explicit. After resolving the singularities of $C_\varepsilon$ by blow-ups, we obtain a total of three distinguished lines: the total transform of the lines $\langle n_0,n_u \rangle$ and, if $H_\eta$ has higher order of contact, exceptional divisors not intersecting the proper transform of $C_\varepsilon$. These three lines are contracted to points $u \le \eta$, with $\ell$ lines contracting to $u$ if $\ell u \le \eta$ but $(\ell+1)u \not\le \eta$. The exceptional divisor over $n_0$ maps to $H_\eta \cap Q_C$ and the distinguished line $L_0$ maps to the node of~$Q_C$. 

As in the other cases of $(C,\varepsilon)$, four curves are contracted to four singular points of $\Gamma_\varepsilon$, counted appropriately. Moreover, as $(C,\varepsilon)$ is deformed, $\Gamma_\varepsilon$ deforms into a Cayley cubic whose four singularities come out of these four special points, namely the contact points of $H_\eta$ and the node of $Q_C$.

An important distinction between the odd case and remaining cases is that there is no natural double cover of the cone (or its blow-up $S_\varepsilon$) which induces the double cover of $C$ corresponding to $\varepsilon$. Also, the usual trick of recovering the symmetric cubic $\Gamma_\varepsilon$ from the adjugate of $\xc_\varepsilon$ does not work here, as $\adj(\xc_\varepsilon)$ maps $\pp$ onto a circle so we can only recover $Q_C$ and not $H_\eta$. 

\begin{remark}
  Since $H_\eta$ is determined by $\xc_\varepsilon$, we can also determine $\xq_\varepsilon$ from $\xc_\varepsilon$. In computations, this amounts to augmenting $\adj(\xc_\varepsilon)$ with the square of a linear form cutting out $L_0$. 
\end{remark}

\subsection{Singularities of $C_\varepsilon$ when $\varepsilon$ is odd}\label{sec:odd_singularities} 

\begin{proposition}\label{prop:triple_point}
  The curve $C_\varepsilon$ has a point of multiplicity three if and only if $\varepsilon$ is odd.
\end{proposition}

\begin{proof}
  There are no points of multiplicity four or higher on the sextic $C_\varepsilon$ since $C$ is not hyperelliptic. Suppose $n \in C_\varepsilon$ is a triple point and let $D = \rho_\varepsilon^*(n)$ be the preimage. Projecting from the triple point $n$ gives the trigonal pencil $\cl := \omega_C\otimes\varepsilon(-D)$. Observe that $\cl \otimes \varepsilon \simeq \omega_C(-D)$ is effective and apply Lemma~\ref{lem:twisted_pencils}.

  Conversely, if $\omega_C \otimes \varepsilon \simeq \theta \otimes \eta$ where $\theta$ is a vanishing theta-null and $\eta$ is an odd theta characteristic, then let $D$ be the divisor in $|\eta|$. Now $\omega_C\otimes\varepsilon(-D) \simeq \theta$ has two sections. Therefore, $\rho_\varepsilon$ contracts $D$ to a node of multiplicity three.
\end{proof}

\begin{lemma}\label{lem:twisted_pencils} If $\cl$ is a trigonal pencil on $C$ and $\varepsilon \in \jac_C[2]$ is an element of order two for which $|\cl \otimes \varepsilon| \neq \emptyset$ then $\cl$ is a vanishing theta-null and $\varepsilon$ is odd.
\end{lemma}
\begin{proof}
  We have $\deg(\cl)=3$ and $h^0(\cl)=h^0(\omega_C\otimes\cl^\vee)=2$. If $\cl^{\otimes 2}\not\simeq\omega_C$ then $h^0(\cl^{\otimes 2})=3$ and $\sym^2 \H^0(\cl) \isoto \H^0(\cl^{\otimes 2})$.
From $|\cl^{\otimes2}|$ we obtain a model of $C$ as a trigonal cover of a conic $Q\subset\PP H^0(\cl^{\otimes 2})$.
For each $D \in |\cl\otimes \varepsilon|$ we have $2D \in |\cl^{\otimes 2}|$. However, if a line is not tangential to $Q$ then it has no chance of inducing a divisor divisible by two since the map $C \to Q$ is of degree 3. A tangent line to $Q$ does pull back to an even divisor $2D'$, but then $D' \in |\cl|$. As a result, $|\cl\otimes\varepsilon| = \emptyset$.
\end{proof}

\begin{corollary}\label{cor:odd_bielliptic}
  The pair $(C,\varepsilon)$ cannot be both bielliptic and odd.
\end{corollary}
\begin{proof}
   The Prym-canonical image of a bielliptic pair is a smooth cubic whereas the Prym canonical image of an odd pair has a triple point.
\end{proof}

For the rest of this section we assume $(C,\varepsilon)$ is odd with $\theta$ the vanishing theta-null and $\eta=\theta \otimes \varepsilon$ with $H_\eta$ the corresponding tritangent. 

\begin{lemma}\label{lem:vertex_avoidance}
  The tritangent $H_\eta$ does not pass through the vertex of $Q_C$.
\end{lemma}
\begin{proof}
  Suppose it does not. Then $H_\eta \cdot Q_C = L_1 + L_2$ where $L_i$ are distinct lines. The point $L_1 \cap L_2$ is the vertex of $Q_C$ and is not contained in $C$. Since $H_\eta \cdot C = (L_1 + L_2) \cdot_{Q_C} C$ is an even divisor, each $L_i \cdot_{Q_C} C$ must be even, which is a contradiction.
\end{proof}

\begin{lemma}\label{lem:node_bijection}
  The singularities of $C_\varepsilon$ that are distinct from the triple point $n_0$ are in bijection with the contact points of the tritangent $H_\eta$. 
\end{lemma}
\begin{proof}
  Let $n \in C_\varepsilon \setminus n_0$ be a singular point. The line $L:=\langle n_0,n \rangle$ pulls back on $C$ to a divisor $D + \eta$ where $D \in |\theta|$ as in Proposition~\ref{prop:triple_point}. If $x,y \in C$ are such that $\rho_\varepsilon^*(n)\ge x+y$ and we set $z\in C$ so that $x+y+z \equiv \theta$ then $2= h^0(\theta\otimes\eta-x-y)=h^0(\eta+z)$, which implies $z \le \eta$. 

  Conversely, suppose $z \le \eta$ and let $D \in |\theta|$ be the unique divisor such that $D \ge z$. We have $h^0(\theta\otimes\eta-(D-z))=h^0(\eta+z)=2$ which implies that the divisor $D-z$ maps to a singular point $n$ of $C_\varepsilon$. The points on $D-z$ lie on a ray of the cone $Q_C$ so that $D-z \neq \eta$ as a result of Lemma~\ref{lem:vertex_avoidance}. Since $\rho^*_\varepsilon(n_0)\equiv \eta$ we have $n \neq n_0$.
\end{proof}

For $u \le \eta$ let $n_u \in C_\varepsilon$ denote the corresponding singularity, that is $\rho_\varepsilon^*(n_u)+u \in |\theta|$. 

\begin{remark}\label{rem:direction_at_n0}
  The proof of Lemma~\ref{lem:node_bijection} implies that $n_u$ lies on the line maximally tangent to the branch of $C_\varepsilon$ at $n_0$ containing $z$. 
\end{remark}

\begin{lemma}
  If $\ell u \le \eta$ and $(\ell+1)u \not\le \eta$ then $n_u$ is of type $A_{2\ell-1}$ or $A_{2\ell}$. The ray of $Q_C$ containing $u$ is tangential to $C$ if and only if $n_u$ is of type $A_{2\ell}$.
\end{lemma}
\begin{proof}
  Recall from Lemma~\ref{lem:mult_seq} that the multiplicity sequence of $C_\varepsilon$ is $(3,2,2,2)$. In light of Lemma~\ref{lem:node_bijection}, this finishes the proof if $H_\eta$ has three distinct contact points.

  In general, fix $u\le \eta$ and let $\rho_\varepsilon^*(n_u)=x+y$. The multiplicity sequence of $n_u$ consists only of $2$'s and we look for higher tangents. Let $D := \omega\otimes \varepsilon - 2x-2y$ and $E = \omega_C - D$. Since $x+y+u \equiv \theta$ we have $E \equiv \eta-u+x+y$, which is effective. Obviously, $D$ is effective precisely when $E$ is dominated by a canonical divisor. There exists a unique ray $L$ of $Q_C$ satisfying $L\cdot C \ge x+y$ and then $L\cdot C = x+y+u$. There exists a unique line $L'$, necessarily contained in $H_\eta$, such that $L'\cdot C \ge \eta-u$. Then $E \le \omega_C$ if and only if $L$ and $L'$ intersect.

  We claim $L$ and $L'$ intersect precisely when $u \le \eta-u$. If $u \le \eta-u$ then $u\in L'$ and thus $L\cap L' \neq \emptyset$. If $L \cap L' \neq \emptyset$ then $L\cap H_\eta =u$ implies $u \in L'$. Assuming $u \not \le \eta-u$ we get $L' \cdot C \ge \eta$, which is a contradiction as there is a unique hyperplane $H$ with $H\cdot C \ge \eta$.  
  
  Therefore, $n_u$ admits a higher tangent ($D \ge 0$) precisely when $2u \le \eta$. It is clear from the previous argument that $3u \le \eta$ if and only if the higher tangent has contact order $6$ at $n_u$. The singularity $n_u$ being unibranch is equivalent to $x=y$, which implies $L$ is tangent to $C$ at $x$.
\end{proof}

\begin{lemma}
  There exists a unique line $L_0 \subset \pp$ such that $L_0 \cap C_\varepsilon = \{n_u \mid u \le \eta_u\}$.
\end{lemma}
\begin{proof}
  Write $\eta = \sum_{i} \ell_i u_i$ where $\ell_i \in \{1,2,3\}$ and let $\theta-u_i = x_i+y_i$. Then $3\theta\equiv \sum_i \ell_i(x_i+y_i+u_i)$ so that $\omega_C\otimes \varepsilon = 3\theta\otimes \eta^\vee $ contains the divisor $\sum_{i} \ell_i(x_i+y_i)$. Recall $\rho_\varepsilon^*(n_{u_i}) = x_i+y_i$.
\end{proof}

\subsection{Proof of Proposition~\ref{prop:descend_epsilon}}\label{sec:descend_epsilon}

Let $\varepsilon$ be an order two line bundle on $C$ defined over $\ksep$ and suppose that the isomorphism class $[\varepsilon]$ is $\Gal(\ksep/k)$-invariant. We show that a representative for the class $[\varepsilon]$ exists over $k$.

Using Theorem~\ref{thm:prym} we get a cubic symmetroid $\Gamma_\varepsilon\subset\PP^3=|\omega_C|$ over $\ksep$, but since the construction of $\Gamma_\varepsilon$ only depends on the isomorphism class of $\varepsilon$, we see that the cubic $\Gamma_\varepsilon$ is $\Gal(\ksep/k)$-invariant and hence defined over $k$.

The construction of $\Gamma_\varepsilon$ canonically identifies $\PP \H^0(\omega_C\otimes \varepsilon)$ with the domain of the symmetrization map $\xq_\ca$ (see Section~\ref{sec:gauss_map}). From  Proposition~\ref{prop:intro_k_symmetrization} it follows that for non-degenerate $\Gamma_\varepsilon$, the symmetrization map is defined over $k$ as well, and that the domain is in fact isomorphic to $\PP^2$ over $k$. As a result, $\PP \H^0(\omega\otimes \varepsilon)$ has $k$-rational points, which means that $|\omega_C\otimes\varepsilon|$ contains effective divisors that are defined over $k$. It follows that we can assume $\varepsilon$ is defined over $k$.
  
When $\varepsilon$ is bielliptic (see Definition~\ref{def:bielliptic}) we get degenerate cubics $\Gamma_\varepsilon$ (Section~\ref{sec:degenerate_cubics}). In this case, $\Gamma_\varepsilon$ is a cone over a cubic plane curve $E \subset \pp$ defined over $k$. Let $\ell$ be the degree three divisor class on $E$ corresponding to the pullback of a line on $\pp$. Observe that $\jac_E$ contains a $k$-rational point $[\varepsilon']$, this is the class which pulls back to $[\varepsilon]$ on $C$. Define the map $\psi\colon E\to\Jac_E$ given by $p\mapsto [3p-\ell]$. Obtain the degree 0 divisor $D = \psi^*([\varepsilon']-[\co_E])$, which is defined over $k$. We see that $[3D] \equiv \psi_*(D) \equiv 9([\varepsilon']-[\co_E]) \equiv [\varepsilon']-[\co_E]$. As a consequence, $\co_E(3D)$ represents the class $[\varepsilon']$. Pulling back the line bundle $\co_E(3D)$ to $C$ gives a representative of $[\varepsilon]$.

\section{Constructing the genus three curve}\label{sec:prym_construction}

Let $C\subset \ppp$ be a canonically embedded curve of genus four with $\varepsilon$ a line bundle on $C$ of order two. We give an explicit construction of $X_\varepsilon$ with the pair of tetragonal pencils $(\cl_1,\cl_2)$.

Let $\xq_\varepsilon \colon  \pp_\varepsilon \to \vppp$ be the symmetrization corresponding to $\varepsilon$ obtained from Theorem~\ref{thm:prym} and let $\vGam_\varepsilon=\xq_\varepsilon(\pp)$ be the strict dual of $\Gamma_\varepsilon$. The quadric containing $C$ is $Q_C$ and $\vQ_C \subset \vppp$ is the dual of $Q_C$. When $Q_C$ is singular, $\vQ_C$ denotes the plane dual to the vertex of $Q_C$ and we write $\vq_C \subset \vQ_C$ for the conic parametrizing the enveloping planes of $Q_C$.

\begin{theorem}\label{thm:main_prym}
  If $\varepsilon$ is not odd, then the image of $X_\varepsilon$ under its canonical map is the pullback $\xq_\varepsilon\inv(\vQ_C)$ of the dual quadric $\vQ_C$ by the symmetrization $\xq_\varepsilon$. 
\end{theorem}
\begin{remark}
  We do not exclude the possibility that $X_\varepsilon$ is hyperelliptic, whose canonical image is a conic. This happens precisely when $\varepsilon$ is even. In this case, the eight branch points on the conic $\xq_\varepsilon\inv(\vQ_C)$ are $\xq_\varepsilon\inv(\vq_C)$.
\end{remark}

It remains to identify the two tetragonal pencils $(\cl_1,\cl_2)$ on $X$ corresponding to $(C,\varepsilon)$. In fact, these pencils play a central role in the proof of Theorem~\ref{thm:main_prym}. 

\begin{theorem}\label{thm:prym_construction}
  If $Q_C$ is smooth then the two linear systems $|\cl_1|$ and $|\cl_2|$ on $X_\varepsilon$ are obtained by pulling back the two rulings of $\vQ_C$.
\end{theorem}

When $Q_C$ is singular there are some arithmetic subtleties which we deal with in Section~\ref{sec:subtle_arithmetic}. For now assume $k=\overline{k}$, let $S$ be the double cover of the plane $\vQ_C$ branched over $\vq_C$ and note $S \simeq \p \times \p$. Let $\overline{Y} = \vGam_\varepsilon \cap \vQ_C$ and $Y$ the preimage of $\overline{Y}$ in $S$.

\begin{theorem}\label{thm:prym_construction_even}
  If $Q_C$ is singular and $\Gamma_\varepsilon$ is non-degenerate ($\varepsilon$ is even, not bielliptic) then $X_\varepsilon$ is the normalization of $Y$. The two tetragonal pencils $(\cl_1,\cl_2)$ on $X_\varepsilon$ are induced by the two rulings of $S$. The pencils $\cl_1$ and $\cl_2$ are conjugated by the hyperelliptic involution on $X_\varepsilon$.
\end{theorem}

\begin{theorem}
  If $Q_C$ is singular and $\Gamma_\varepsilon$ is a cone ($\varepsilon$ is even and bielliptic) then $X_\varepsilon \simeq Y \times_{\overline{Y}} \xq_\varepsilon\inv(\overline{Y})$. The two pencils $\cl_1\simeq \cl_2$ induced from the rulings of $S$ coincide as $Y$ is the diagonal of $S$.
\end{theorem}

\subsection{Strategy of proof}\label{sec:strategy}
We sketch the strategy for smooth $Q_C$ and $\Gamma_\varepsilon$ not a cone, the other cases being analogous. The Jacobian of the pullback $X := \xq^{-1}_\varepsilon(\vQ_C)$ is the Prym variety of $C$ if $X$ is reduced and the pair $(C,X)$ fit into Recillas' trigonal construction (see~\cite{Recillas1974} or~\cite{donagi}*{\S 2.4} and Proposition~\ref{P:trigonal_construction}). 

A ruling of $Q_C$ induces a trigonal map $f_C\colon C \to L$ and the dual ruling in $\vQ_C$ induces a tetragonal pencil $f_X\colon X \to L$ obtained by pulling back the lines in $\vQ_C$ via $\xq_\varepsilon$. Notice that we canonically identified the target $L$'s of $f_C$ and $f_X$.

Any point $x \in C$, by the virtue of being in $\Gamma_\varepsilon$, will induce a singular conic $q_x$ on $\pp$. We will prove that the two components of this singular conic partition the fiber $f_X\inv(f_C(p))$ into two sets of two. This establishes that $(C,X)$ fit into Recillas' trigonal construction.

It remains to show that the double cover $\tilde C_\varepsilon \to C$ induced by $\varepsilon$ is obtained by labeling the partitions induced on the fibers of $f_X$. Equivalently, we may show $\tilde C_\varepsilon$ marks the components of the singular conics parametrized by $C$. Theorem~\ref{thm:induced_double_cover} states that the double cover $\tilde C_\varepsilon \to C$ is induced by the natural double cover of the cubic symmetroid $\Gamma_\varepsilon$. Section~\ref{sec:natural_double_cover} proves that this cover is obtained by labeling the components of the singular conics parametrized by $\Gamma_\varepsilon$. This concludes the outline of proof. 

\subsection{General case} \label{sec:prym--general} 

We assume $(C,\varepsilon)$ is not bielliptic, odd or even. 

\begin{lemma}
  The curve $X = \xq_\varepsilon\inv(\vQ_C)$ is reduced.
\end{lemma}
\begin{proof}
If $X$ is non-reduced, we can find a singular point $p \in X$ where $\xc_\varepsilon$ is well defined and the differential $\dd \xq_\varepsilon|_p$ is of full rank.  Then the plane $H_{\xq} := \im \dd \xq_\varepsilon |_p$ must be tangential to $\vQ_C$ at $\xq_\varepsilon(p)$. The identity $\xq_\varepsilon = \gamma_\varepsilon \circ \xc_\varepsilon$ (Lemma~\ref{lem:gauss}) implies $\xq_\varepsilon(p)$ is dual to the tangent plane $H_{\xc}$ of $\Gamma_\varepsilon$ at $\xc_\varepsilon(p)$.  Since $H_{\xq}$ is tangential to $\xq_\varepsilon(\pp)$ at $\xq_\varepsilon(p)$, $\xc_\varepsilon(p)$ is dual to $H_\xq$ by the reflexivity theorem.  With $H_\xq$ tangent to $\vQ_C$ at $\xq_\varepsilon(p)$, $H_\xc$ must be tangent to $Q_C$ at $\xc_\varepsilon(p)$. This contradicts that $C$ is a smooth complete intersection of $Q_C$ and $\Gamma_\varepsilon$ (Corollary~\ref{cor:odd_is_evil}).
\end{proof}

\begin{lemma}\label{lem:partition}
  General $x \in C$ naturally partitions $f_X\inv(f_C(x)) \subset X$ into two sets of two.
  
\end{lemma}
\begin{proof}
  Let $L_1,L_2 \subset Q_C$ be the two lines of $Q_C$ passing through $x$ with duals $\vL_1,\vL_2 \subset \vQ_C$. Up to relabeling, $f_C^*f_C(x)=L_1 \cdot C$ and $f_X^*f_C(x) = X \cdot \xq_\varepsilon^*(\vL_1)$.

 Let $H_x \subset \vppp$ be the plane dual to $x$. The conic$\xq_\varepsilon\inv(H_x) \subset \pp_\varepsilon$ induces on $X$ the degree eight divisor $D_1+D_2$ where $D_i := X \cdot \xq_\varepsilon\inv(\vL_i)$. Since $x\in \Gamma_\varepsilon$, the pullback $\xq_\varepsilon\inv(H_x)$ is a \emph{singular} conic $\ell_1 \cup \ell_2$. Therefore, $D_1+D_2$ is partitioned into two, namely $\ell_1 \cdot X$ and $\ell_2\cdot X$. We claim $\ell_i$'s contain precisely two points in each $D_i$. Indeed, we may view $\xq_\varepsilon(l_i)$ as a (possibly non-reduced) conic in $H_x$ which can only intersect $\vL_j$ at two points.
\end{proof}

\subsection{Bielliptic non-even case}\label{sec:bielliptic_case} Assuming $\varepsilon$ is bielliptic but not even, $\Gamma_\varepsilon$ is a cone but $Q_C$ is not. Then $\xq_\varepsilon$ realizes $\pp$ as a a fourfold cover of the plane $H_\varepsilon$ dual to the vertex of~$\Gamma_\varepsilon$. 

\begin{lemma}
  The conic $\vQ_C \cap H_\varepsilon$ is smooth and $X = \xq_\varepsilon\inv(\vQ_C)$ is reduced. 
\end{lemma}
\begin{proof}
  If the intersection is singular then $H_\varepsilon$ is tangential to $\vQ_C$. Then the vertex of $\Gamma_\varepsilon$ is on $Q_C$, in which case $C = Q_C \cap \Gamma_\varepsilon$ would have to be singular. 

  The map $\xq_\varepsilon\colon  \pp_\varepsilon \to H_\varepsilon$ is branched along a smooth
  cubic in $H_\varepsilon$. Then $X$ must be reduced since the conic $\vQ_C \cap H_\varepsilon$ cannot be contained in the branch locus.
\end{proof}

\noindent The map $\xq_\varepsilon|_X \colon X \to \vQ_C \cap H_\varepsilon$ is a tetragonal pencil, which we will denote by $f_X$. 

\begin{lemma}
  General $x \in C$ naturally partitions $f_X\inv(f_C(x)) \subset X$ into two sets of two.
\end{lemma}
\begin{proof}
  There is a line $L \subset Q_C$ for which $f_C^*f_C(x) = L \cdot C$. Then $D := \xq_\varepsilon\inv(p) \cdot X$ is the fiber $f_X^*f_C(x)$. Since $x \in \Gamma_\varepsilon$, the pullback $\xq_\varepsilon\inv(H_x)$ is the union of two distinct lines $\ell_1$ and $\ell_2$. As in Lemma~\ref{lem:partition} we see that $X\cdot \ell_i$ induces the desired partition of $D$.
\end{proof}

The argument that $X$ corresponds to $\varepsilon$ is analogous to the one given in Section~\ref{sec:strategy}. This concludes the proof of Theorem~\ref{thm:main_prym} for this case. We note the following.

\begin{lemma}\label{lem:alternative_Q}
  Projecting from the vertex of $\Gamma_\varepsilon$ realizes the quadric $Q_C$ as the double cover of the plane $H_\varepsilon^\vee$ branched over the smooth conic dual to $\xq_\varepsilon(X) \subset H_\varepsilon$.
\end{lemma}
\begin{proof}
  By construction of $X$, we have $\xq_\varepsilon(X) = H_\varepsilon \cap \vQ_C$. Note that $H_\varepsilon$ is the plane dual to the vertex of $\Gamma_\varepsilon$.  The projection of $Q_C$ from this vertex gives a branch locus in $H_\varepsilon^\vee$ dual to $H_\varepsilon \cap \vQ_C$.
\end{proof}

\subsection{Even non-bielliptic case} \label{sec:even_case}

Assume $(C,\varepsilon)$ is even but not bielliptic. Thus, $Q_C$ is a cone but $\Gamma_\varepsilon$ is not. The dual $\vQ_C$ is a plane containing $\vq_C$.

\begin{lemma}\label{lem:reduced8}\label{lem:Xbar_is_smooth}
  The conic $\overline{X}=\xq_\varepsilon\inv(\vQ_C)$ is smooth and $B:=\xq_\varepsilon\inv(\vq_C)$ is reduced.
\end{lemma}
\begin{proof}
  The vertex of $Q_C$ cannot lie on $\Gamma_\varepsilon$ so the pullback $\xq_\varepsilon\inv(\vQ_C)$ is smooth. If $B$ is not reduced, $\vq_C$ is tangential to the dual $\vGam_\varepsilon$ of the cubic symmetroid $\Gamma_\varepsilon$. By the reflexivity theorem, $\Gamma_\varepsilon$ is tangential to $Q_C$ and $C$ is singular.
\end{proof}

The conic $\overline{X}$ has eight distinguished points $B$ and we can construct a hyperelliptic curve $X \to \overline{X}$ as the double cover of $\overline{X}$ branched over $B$. We show that the Jacobian of $X$ is the Prym variety determined by $(C,\varepsilon)$. Consider a double cover $S$ of the plane $\vQ_C$ branched over the conic $\vq_C$. Write $\overline{Y}$ for the singular plane quartic $\xq_\varepsilon(\overline{X}) = \vQ_C \cap \vGam_\varepsilon$. Let $Y \subset S$ be the preimage of $\overline{Y}$. 

\begin{lemma}
  The curve $Y$ is irreducible and of geometric genus three.
\end{lemma}
\begin{proof}
The double cover $S \to \vQ_C$ is just a projection map therefore $Y$ is a $(4,4)$-curve in $S \simeq \p\times\p$ of arithmetic genus nine.

  Counting with multiplicities, $\overline{Y}$ has three singularities since it is a degree four plane curve of geometric genus 0. If $Y$ were reducible, it would have to break into two isomorphic copies of $\overline{X}$, each embedded as a $(2,2)$-curve. But a $(2,2)$-curve in $\p\times\p$ has arithmetic genus one and, therefore, cannot support three singularities. This proves $Y$ is irreducible.
   With six singularities, this makes $Y$ a genus three curve ramified exactly over the eight distinguished points of $\overline{Y}$. 
\end{proof}

Define $X$ to be the desingularization of $Y$, so that $X \to \overline{X}$ is a hyperelliptic curve branched over $B$. The two rulings of $S$ induce a pair $(\cl_1,\cl_2)$ of tetragonal pencils on $X$. It is clear that the target of the trigonal pencil of $C$ is identified with the target of both of these tetragonal pencils on $X$. Let $\iota\colon X \to X$ denote the hyperelliptic involution.

\begin{lemma}\label{lem:conjugacy}
  The two pencils $\cl_1$ and $\cl_2$ are conjugate under the hyperelliptic involution. Moreover, for any $p \in X$ and any $D \in |\cl_i|$ we have $p + \iota(p) \not\le D$.
\end{lemma}
\begin{proof}
  For any $\vL \subset \vQ_C$ tangent to $\vq_C$, let $\vL_1$ and $\vL_2$ stand for the two lines in $S$ above $\vL$. The intersections $Y \cdot \vL_i$ must be conjugate as they map to the same four points $\vL \cdot \overline{Y}$. It is clear that a base-point-free tetragonal pencil cannot contain a pair of conjugate points.
\end{proof}

Let us fix one of the rulings of $S$ and consider the associated tetragonal pencil $f_{X} \colon X \to \p$. In light of Section~\ref{sec:strategy}, it remains to prove the following result.

\begin{lemma}
  Every $x \in C$ partitions the fiber $f_{X}\inv(f_C(x)) \subset X$ into two sets of two.
\end{lemma}
\begin{proof}
  A point $x \in C$ is contained in a ray $L \subset Q_C$ which has dual $\vL \in \vQ_C$, which is a tangent of $\vq_C$. Let $D$ be the divisor on $\overline{X}$ induced by $\vL$. Equivalently, $D = \overline{X} \cdot \xq_\varepsilon(H_x)$, where $H_x\subset \vppp$ is the plane dual to $x$. However, since $x \in \Gamma_\varepsilon$ the pullback of $H_x$ breaks into two lines, each of which intersects $\overline{X}$ at two points.

  Let $(D_1,D_2) \in |\cl_1|\times |\cl_2|$ denote the pair of divisor induced by pulling back $\vL$ to $S$. Then $D_i$ lies over $D$ and by Lemma~\ref{lem:conjugacy} no two points of $D_i$ map to the same points in $D$. As a result, the partitioning induced on $D$ induces a partition on $D_i$ into two sets of two.
\end{proof}

\subsection{Even and bielliptic case} \label{sec:prym_even_and_bi}

We now assume that $(C,\varepsilon)$ is even and bielliptic. In this case $\xq_\varepsilon(\pp_\varepsilon)$ is a plane $H_\varepsilon \subset \vppp$ which is transversal to the plane $\vQ_C$.  

\begin{lemma}
  The curve $\overline{X}:=\xq_\varepsilon\inv(\vQ_C)$ is a smooth conic. The eight points $B:=\xq_\varepsilon\inv(\vq_C) \subset \overline{X}$ are all distinct.
\end{lemma}
\begin{proof}
  For the first claim, the proof of Lemma~\ref{lem:Xbar_is_smooth} applies verbatim. For the second, we need to show that $H_\varepsilon$ is not tangential to $\vq_C$. Indeed, if $p\in\vq_C \cap H_\varepsilon$ were a point of tangency, then the vertex of $\Gamma_\varepsilon$ would be on a ray of $Q_C$ and $C$ would not be smooth.
\end{proof}

The line $\overline{Y} = \vQ_C \cap H_\varepsilon$ has two distinguished points $\{p_1,p_2\}=\overline{Y} \cap \vq_C$. Consider the double cover $Y$ of $\overline{Y}$ branched over $p_1,p_2$. Observe that $X$ fits into the following Cartesian diagram:
\begin{equation}\label{eq:even_bi--X}
  \begin{tikzcd}
    X\arrow[dr, phantom, "\ulcorner", very near start] \arrow[r,"f"] \arrow[d,"2:1"'] & Y \arrow[d,"2:1"] \\
    \overline{X} \arrow[r, "4:1"] & \overline{Y}
  \end{tikzcd}
\end{equation}
As $Y$ is of genus zero, the map $f\colon X \to Y$ is a tetragonal pencil. Let $\cl = f^* \co_{Y}(1)$. 

\begin{lemma}
  The line bundle $\cl$ is invariant under the involution of $X$, i.e., $\cl \simeq \iota^* \cl$. Moreover, $\cl^{\otimes 2} \simeq \omega_X^{\otimes 2}$ but $\cl \not\simeq \omega_X$.
\end{lemma}
\begin{proof}
  The involution $\iota\colon X \to X$ commutes with the involution of $Y$ over $\overline{Y}$. Since the line bundle $\co_Y(1)$ is invariant under this involution, so is its pullback $\cl$. 
  
  The line bundle $\cl^{\otimes 2}$ is the pullback of $\co_{Y}(2)$, which in turn is the pullback of $\co_{\overline{Y}}(1)$. The pullback of $\co_{\overline{Y}}(1)$ to $\overline{X}$ is $\co_{\overline{X}}(2)$ since the map $\xq_\varepsilon$ is defined by quadrics. Moreover, the pullback of $\co_{\overline{X}}(1)$ to $X$ is $\omega_X$. This proves the second statement. 
  
  Finally, if $\cl \simeq \omega_X$ then the map $X \to Y$ must be the composition of the canonical map $X \to \overline{X} \toi \pp$ with a projection $\pp \ratto \p$. On the other hand, the eight Weierstrass points on $X$ all get mapped to the two ramification points of $Y$ over $\overline{Y}$. This means that the point of projection $\pp \ratto \p$ must be the intersection of two lines each of which contain four branch points of $X \to \overline{X}$. This is impossible since $\overline{X}$ is of degree two in $\pp$.
\end{proof}

\begin{proposition}
  The triple $(X,\cl,\cl)$ corresponds to $(C,\varepsilon)$ via the trigonal construction.
\end{proposition}
\begin{proof}
  As in the previous cases, we will show that a general point $p \in C$ induces a pair of divisors $(D_1,D_2) \in |\cl| \times |\cl|$ and a $(2,2)$-partition of each $D_i$. The dual plane $H_p$ of $p$ intersects $\overline{Y}$ at a single point $q$ with preimages $q_1,q_2 \in Y$. The preimages of $q_1$ and $q_2$ in $X$ give a pair of conjugate divisors, say $D_1$ and $D_2$. 
  
  Since $p$ is a smooth point of $\Gamma_\varepsilon$, the pullback of $H_p$ is a union of two distinct lines $L_1,L_2 \subset \pp$ each intersecting $\overline{X}$ in two points. Both $D_1$ and $D_2$ lie over these four points in $\overline{X}$ hence both $D_1$ and $D_2$ are partitioned.
\end{proof}

\noindent We make the following observation for future reference. 

\begin{lemma}\label{lem:even_bi_Q}
  The projection of the cone $Q_C$ from the vertex of $\Gamma_\varepsilon$ ramifies over two distinct lines $L_1,L_2$. The duals $p_1,p_2$ of $L_1,L_2$ are the branch points of $Y \to \overline{Y}$ given in \eqref{eq:even_bi--X}. 
\end{lemma}
\begin{proof}
  Projection from the vertex of $\Gamma_\varepsilon$ naturally realizes $Q_C$ as a double cover of the plane $H_\varepsilon^\vee$ dual to $H_\varepsilon$. Let $R_i$ be the ray in $Q_C$ mapping to $L_i$. Note that the tangent plane of a ray $R \subset Q_C$ contains the vertex of $\Gamma_\varepsilon$ iff $R=R_1$ or $R_2$.

  The tangent planes of the rays are parametrized by $\vq_C$ and the planes containing the vertex of $\Gamma_\varepsilon$ are parametrized by $H_\varepsilon$. Therefore, $\{p_1,p_2\} = H_\varepsilon\cap \vq_C$ represent the tangent planes $H_i$ of $R_i$. The lines $L_1,L_2 \subset H_\varepsilon^\vee$ are projections of $H_1, H_2$ and therefore are dual to the points $p_1,p_2 \in H_\varepsilon$. Recall $H_\varepsilon \cap \vq_C$ are the branch points of $\overline{Y} \to Y$.
\end{proof}

\subsection{Arithmetic subtleties in the even case}\label{sec:subtle_arithmetic}

When $C$ has a vanishing theta-null $\theta$ over $\ksep$ and $[\varepsilon]$ is even, we see that $\Gamma_\varepsilon$ and $Q_C$ determine the conic $\overline{X}_\varepsilon$ but not any particular twist of the hyperelliptic curve $X_\varepsilon$. We now show that the correct twist may nevertheless be determined using the given symmetrization $q_\varepsilon$ of $\Gamma_\varepsilon$ and paying attention to fields of definition during the constructions.

The quadric $Q_C$ is a cone over a conic $q_C$ and projection from the vertex of $Q_C$ gives the unique trigonal map $C\to q_C$. The conic $q_C$ is the Brauer--Severi variety $\bs_{[\theta]}$. Let us first consider the case where $[\varepsilon]$ is not bielliptic, where we obtain the genus $0$ curve $\overline{Y}=\vQ_C\cap\vGam_\varepsilon$ (Section~\ref{sec:even_case}). Consider a double cover $S$ of the plane $\vQ_C$ branched over $\vq_C$. The surface $S$ admits a smooth quadric model in $\PP^3$: if $q(x_0,x_1,x_2)=0$ is a model for $q_C$, then $x_3^2-q(x_0,x_1,x_2)$ provides a model for $S$. Observe that the surface $S$ is isomorphic to $\PP^1\times\PP^1$ over $\kbar$.

Let $\tC\to C$ be the double cover obtained by pulling back $C$ along the distinguished double cover of $\Gamma_\varepsilon$ induced by its symmetrization as described in Section~\ref{sec:natural_double_cover}. Similarly, there is a distinguished choice of $S$ over $\vQ_C$: the one that has split rulings. It is the one for which the space quadric model has square discriminant, which we can attain by scaling $q$ appropriately. In $S$ we find the double cover $Y\to\overline{Y}$ branched over $\vq_C\cap \overline{Y}$. This gives us models over the base field of the required curves to let the construction of Section~\ref{sec:even_case} go through. The fact that $S$ has split rulings ensures our selected twists are trigonally related.

\begin{remark}
  When $X$ is hyperelliptic, unlike the general case, the Prym varieties of twists of the distinguished cover $\tC\to C$ are also Jacobians: the corresponding curves are the quadratic twists of $X\to \overline{X}$.
\end{remark}

If $[\varepsilon]$ is both even and bielliptic, the map $\xq_\varepsilon\colon \PP^2_\varepsilon\to H_\varepsilon$ expresses the canonical model $\overline{X}$ as a tetragonal cover of $\overline{Y}=H_\varepsilon\cap \vQ_C$. We can follow the non-bielliptic case and construct the double cover $S\to\vQ_C$, branched along $q_C$, with split rulings. This gives us a double cover $Y\to\overline{Y}$, and we obtain $X= \overline{X}\times_{\overline{Y}}Y$.

\section{Reverse construction from genus three curves}\label{sec:reverse}

In this section we prove Theorem~\ref{thm:reconstruction}. We continue with the notation from Section~\ref{sec:reverse_from_intro}. If the intersection $\Gamma_{X,\kappa} \cap Q_{X,\kappa}$ happens to be a curve of geometric genus four, then the construction in Section~\ref{sec:prym_construction} clearly recovers $(X,\kappa)$. We will prove the converse: if $(X,\kappa)$ is obtained from $(C,\varepsilon)$, we will show that the intersection $\Gamma_{X,\kappa} \cap Q_{X,\kappa}$ is the canonical model of $C$ and the double cover of $\Gamma_{X,\kappa}$ associated to the symmetrization $\xq_{X,\kappa}$ recovers $\varepsilon$ on $C$. 

Despite the uniformity of the statement, the proof requires the analysis of four cases, as in Table~\ref{tbl:strata}. We treat each separately in Sections~\ref{sec:reverse-general}-\ref{sec:reverse-hyp-bi}.

Recall that the class $\kappa$ corresponds to an unordered pair of linear equivalence classes $\{[\cl_1],[\cl_2]\}$ of line bundles. We initially assume that $\cl_1$ and $\cl_2$ are defined over $k$ and remove this assumption in Section~\ref{sec:descend_Q_Gamma}.

\subsection{Preliminaries} \label{sec:prelims}

Let $\ppp = \Pp\H^0(\omega_C)$ and $\vppp= |\H^0(\omega_C)|$. Recall $Q_C$ is the quadric containing $C \subset \ppp$. The symmetrization induced by $\varepsilon$ (Section~\ref{sec:prym_canonical_map}) is denoted by  
\[
  \ca_\varepsilon \colon  \H^0(\omega_C)^\vee \to \sym^2\H^0(\omega_C\otimes \varepsilon).
\]
The cubic symmetroid $\Gamma_\varepsilon$ corresponding to $\varepsilon$ is the pullback of the discriminant locus in $|\sym^2\H^0(\omega_C\otimes \varepsilon)|$ to $\ppp=\Pp\H^0(\omega_C) = |\H^0(\omega_C)^\vee|$. 

In Section~\ref{sec:prym_construction} we construct the canonical image of a curve $X$ in $\Pp \H^0(\omega_C\otimes \varepsilon)$. We thus identify $\H^0(\omega_C\otimes \varepsilon)$ and $\H^0(\omega_X)$ canonically up to scaling. The symmetrization map becomes
\begin{equation}
  \label{eq:reconstruction--symmetrization}
  \ca_\varepsilon \colon  \H^0(\omega_C)^\vee \to \sym^2\H^0(\omega_X).
\end{equation}
The cubic symmetroid $\Gamma_\varepsilon$ is still the pullback of the discriminant locus via \eqref{eq:reconstruction--symmetrization}. 

When $X$ is hyperelliptic, its canonical image $\overline{X}$ is a conic in $\pp$. Let $q_{\overline{X}} \in \sym^2 \H^0(\omega_X)$ be the defining equation of $\overline{X}$. We make use of the following simple observations.
\begin{lemma}\label{lem:j1-injective}
  The natural restriction map $j_1\colon \sym^2 \H^0(\omega_X) \to \H^0(\omega_X^{\otimes 2})$ is an isomorphism when $X$ is not hyperelliptic. If $X$ is hyperelliptic, the kernel of $j_1$ is spanned by $q_{\overline{X}}$. \qed
\end{lemma}

\begin{lemma}\label{lem:j2-injective}
  The multiplication map $j_2 \colon  W_1 \otimes W_2 \to \H^0(\omega_X^{\otimes 2})$ is injective if and only if the two pencils $\cl_1$ and $\cl_2$ are not isomorphic. \qed
\end{lemma}

\subsection{Non-hyperelliptic with distinct pencils}\label{sec:reverse-general} 

Assume $X$ is not hyperelliptic and $\cl_1,\cl_2$ are distinct. From Section~\ref{sec:prym_construction} we conclude $Q_C$ is smooth and $\Gamma_\varepsilon$ is non-degenerate.

\begin{proof}[Proof of Theorem~\ref{thm:reconstruction} in Case 1]
  Since $\Gamma_\varepsilon$ is non-degenerate the symmetrization $\ca_\varepsilon$ is injective. Lemma~\ref{lem:j2-injective} implies $j_2$ is injective.
  
  The pencils $\cl_1$ and $\cl_2$ are constructed in Section~\ref{sec:prym--general} by pulling back the rulings of the dual quadric $\vQ_C$ via $\xq_\varepsilon$. For a point $p\in Q_C$ the plane $H_p$ is tangential to $\vQ_C$ and $H_p \cap \vQ_C$ is the union of two lines in $\vQ_C$. Therefore, the symmetrization establishes an isomorphism
  \[
    Q_C \isoto |W_1| \times |W_2| \colon  p \mapsto \xq_\varepsilon(H_p) \cdot X.
  \]

  Since $Q_C$ spans the ambient space $\ppp$, $j_1\circ \ca_\varepsilon(\H^0(\omega_C))^\vee = j_2(W_1 \otimes W_2)$. Thus $\Pp \H^0(\omega_C)$ is canonically identified with $|\MUV|$. Under this identification, $Q_C$ is the pullback of $|W_1| \times |W_2|$ and $\Gamma_\varepsilon$ is, by definition, the pullback of the discriminant locus from $\sym^2 \H^0(\omega_X)$.
\end{proof}

\subsection{Non-hyperelliptic with coincident pencils}\label{sec:reverse-general-bi}  Assume $X$ is not hyperelliptic but $\cl_1\simeq \cl_2$. For the sake of exposition, let $\cl:= \cl_1 = \cl_2$ and $W := W_1 = W_2$.

The multiplication map $j_2 \colon  W\otimes W \to \H^0(\omega_X^{\otimes 2})$ is not an isomorphism and factors through $\sym^2 W \toi \H^0(\omega_X^{\otimes 2})$. We will use the decomposition: \begin{tikzcd}[cramped, sep=small] 0 \arrow[r] & \bigwedge^2 W \arrow[r] & W\otimes W \arrow[r] & \sym^2 W \arrow[r] & 0 .  \end{tikzcd} 

The map $j_1$ is an isomorphism (Lemma~\ref{lem:j1-injective}). Define $T \subset \sym^2 \H^0(\omega_X)$ as the pullback of $\sym^2 W$, so that \eqref{eq:reverse--M} factors as follows:
\begin{equation}\label{eq:reverse--general-bi}
  \begin{tikzcd}
    \MUV \arrow[dr, phantom, "\ulcorner", very near start] \arrow[r,"\sim"]\arrow[d,twoheadrightarrow] & W_1\otimes W_2 \arrow[d,"j_2",twoheadrightarrow]\\
    T \arrow[dr, phantom, "\ulcorner", very near start]\arrow[d,hookrightarrow] \arrow[r,"\sim"] & \sym^2 W \arrow[d,hookrightarrow] \\
    \sym^2 \H^0(\omega_X) \arrow[r,"\sim", "j_1"'] & \H^0(\omega_X^{\otimes 2}).
  \end{tikzcd}
\end{equation}

We may identify $|\MUV|$ with $|W\otimes W|$ and $|T|$ with $|\sym^2 W|$. Note that the cubic $\Gamma_{(X,\cl,\cl)}$ will be a cone over the cubic curve $E \subset |T| \simeq |\sym^2 W|$ obtained by pulling back the discriminant locus $\cd \subset |\sym^2 \H^0(\omega_X)|$. The vertex of $\Gamma_{(X,\cl,\cl)}$ in $|W\otimes W|$ is $[\bigwedge^2 W]$. Our goal is to show that $|W|\times |W|$ cuts out $C$ from within $\Gamma_{(X,\cl,\cl)}$.

\begin{proof}[Proof of Theorem~\ref{thm:reconstruction} in Case 2]
  From Section~\ref{sec:prym_construction} we know $\Gamma_\varepsilon$ is a cone and $Q_C$ is smooth. Let $H_\varepsilon \subset \vppp$ be the plane dual to the vertex of $\Gamma_\varepsilon$. Projecting $Q_C$ from the vertex of $\Gamma_\varepsilon$ realizes $Q_C$ as a double cover of the plane $H_\varepsilon^\vee$ branched over a smooth conic $q \subset H_\varepsilon^\vee$. 
  
  In what follows, we identify $H_\varepsilon^\vee$ with $|\sym^2 W|$. The identification $|\MUV| \isoto  |W\otimes W|$ sends the vertex of $\Gamma_\varepsilon$ to the point $[\bigwedge^2 W]$. It can be checked directly that $|W|\times |W|$ is the unique quadric in $|W\otimes W|$ which when projected from $[\bigwedge^2 W]$ is branched over the second Veronese image $v_2(|W|)$ of $|W|$ within $|\sym^2 W|$. We need to show that $q= v_2(|W|)$. 
  
  From Section~\ref{sec:bielliptic_case} we get $H_\varepsilon=\xq_\varepsilon(\pp)$. From Lemma~\ref{lem:alternative_Q} we get that the dual of $q$ is $\xq_\varepsilon(X)$. The space $T$ defines $\xq_\varepsilon$ so we may identify $H_\varepsilon$ with $\Pp T$ and $H_\varepsilon^\vee$ with $|T| \simeq |\sym^2 W|$. The linear system inducing the restricted map $\xq_\varepsilon|_X$ is $\sym^2 W \to \omega_X^{\otimes 2}$. Therefore, $\xq_\varepsilon|_X$ factors through the map $X \to \Pp W$ induced by the pencil $W \to \cl$ followed by the second Veronese $v_2\colon  \Pp W \to \Pp\sym^2 W$. The dual of the conic $v_2(\Pp W)$ is $v_2(|W|)$, as claimed.
\end{proof}

\subsection{Hyperelliptic with distinct pencils} \label{sec:reverse_hyperelliptic}

Assume $X$ is hyperelliptic, with hyperelliptic involution $\iota\colon X \to X$, and $\cl_1 \not\simeq\cl_2$. Recall $\overline{X}$ is the canonical image of $X$, cut out by $q_{\overline{X}}$. Let $\co_{\overline{X}}(1)=\co_{\pp}(1)|_{\overline{X}}$.

\begin{lemma}\label{lem:reconstruction--invariant}
  The image of $\sym^2 \H^0(\omega_X) \to \H^0(\omega_X^{\otimes 2})$ coincides with the space of $\iota$-invariant $2$-forms $\H^0(\omega_X^{\otimes 2})^\iota$.\qed
\end{lemma}

Let $(W_1\otimes W_2)^\iota$ be the $\iota$-invariant subspace of $ W_1\otimes W_2$. As a consequence of Lemmas~\ref{lem:j2-injective} and~\ref{lem:reconstruction--invariant} the diagram \eqref{eq:reverse--M} factors as follows:
\begin{equation}
  \begin{tikzcd}
    \MUV \arrow[dr, phantom, "\ulcorner", very near start] \arrow[r,twoheadrightarrow]\arrow[d,hookrightarrow] & (W_1\otimes W_2)^\iota \arrow[dr, phantom, "\ulcorner", very near start]\arrow[d,hookrightarrow] \arrow[r,hookrightarrow] & W_1\otimes W_2 \arrow[d,hookrightarrow] \\
    \sym^2 \H^0(\omega_X) \arrow[r,twoheadrightarrow] & \H^0(\omega_X^{\otimes 2})^\iota \arrow[r,hookrightarrow] & \H^0(\omega_X^{\otimes 2}).
  \end{tikzcd}
\end{equation}

\begin{proof}[Proof of Theorem~\ref{thm:reconstruction} in Case 3]

From Section~\ref{sec:prym_construction} we know $\Gamma_\varepsilon$ is non-degenerate ($\ca_\varepsilon$ is injective) and $Q_C$ is a cone. The plane dual to the vertex of $Q_C$ is $\vQ_C$ and $\vq_C \subset \vQ_C$ is the strict dual of $Q_C$. Note, $Q_C$ is naturally a cone over the dual $q_C$ of $\vq_C$.

First we need to identify $\MUV$ with $\H^0(\omega_C)^\vee$. For this we will prove $j_1\circ\ca_\varepsilon(\H^0(\omega_C)^\vee)=j_2((W_1\otimes W_2)^\iota)$. Use $\cl_1,\cl_2$ to construct $X \to \Pp W_1 \times \Pp W_2 \toi \Pp(W_1\otimes W_2)$. From Theorem~\ref{thm:prym_construction_even} we see there is a natural projection from $\Pp(W_1\otimes W_2)$ onto $\vQ_C$ and the ramification locus of the double cover $\Pp W_1 \times \Pp W_2 \to \vQ_C$ is the smooth conic $\vq_C$.

  In Section~\ref{sec:even_case} we saw that the image of $X$ under this projection is $\xq_\varepsilon(\overline{X})$. As this map factors through $\overline{X}$, we must be projecting onto an $\iota$-invariant plane; by dimension reasons, this plane must be $\Pp(W_1\otimes W_2)^\iota$. Since $\xq_\varepsilon|_{\overline{X}}$ is defined by the (degenerate) linear system $\H^0(\omega_C)^\vee \to \co_{\overline{X}}(2)$, we conclude $j_1\circ \ca_\varepsilon(\H^0(\omega_C)^\vee) =j_2((W_1 \otimes W_2)^\vee)$. 

  With $\MUV=\H^0(\omega_C)^\vee$ the cubics $\Gamma_\varepsilon$ and $\Gamma_{X,\kappa}$ coincide as they are both the pullback of the discriminant locus.  We identified $\vQ_C$ with $\Pp(W_1\otimes W_2)^\iota$ and the ramification locus of the double cover $\Pp W_1 \times \Pp W_2 \to \Pp(W_1\otimes W_2)^\iota$ with $\vq_C \subset \vQ_C$. The dual $q_C$ of $\vq_C$ is the intersection of the dual space quadric $|W_1| \times |W_2|$ with the plane $|(W_1\otimes W_2)^\iota|$ inside $|W_1\otimes W_2|$. The pullback of $|W_1|\times |W_2|$ onto $|\MUV|$ gives a cone over $q_C$ with vertex dual to the plane $\vQ_C$; this cone is $Q_C$.
\end{proof}

\subsection{Hyperelliptic with coincident pencils} \label{sec:reverse-hyp-bi}

Now $X$ is hyperelliptic and $\cl:=\cl_1=\cl_2$, $W:=W_1=W_2$. We continue to use $\overline{X}$ and $q_{\overline{X}}$ from previous section.

The map $j_2\colon W\otimes W \to \H^0(\omega_X^{\otimes 2})$ factors through $\sym^2 W$ with kernel $\bigwedge^2 W \toi W \otimes W$. By Lemma~\ref{lem:j1-injective}, the kernel of $j_1$ is generated by $q_{\overline{X}}$ and $j_1$ factors through $\H^0(\omega_X^{\otimes 2})^\iota$ (Lemma~\ref{lem:reconstruction--invariant}). Let $T$ stand for the image of $\MUV$ in $\sym^2 \H^0(\omega_X)$. The diagram \eqref{eq:reverse--M} degenerates to:
\begin{equation}
  \begin{tikzcd}
    \MUV \arrow[dr, phantom, "\ulcorner", very near start] \arrow[r,twoheadrightarrow] \arrow[d,twoheadrightarrow] & (W\otimes W)^\iota \arrow[dr, phantom, "\ulcorner", very near start]\arrow[d,twoheadrightarrow] \arrow[r,hookrightarrow] & W\otimes W \arrow[d,twoheadrightarrow]\\
    T \arrow[dr, phantom, "\ulcorner", very near start] \arrow[r,twoheadrightarrow]\arrow[d,hookrightarrow] & (\sym^2 W)^\iota \arrow[d,hookrightarrow] \arrow[dr, phantom, "\ulcorner", very near start]\arrow[r,hookrightarrow] & \sym^2 W \arrow[d,hookrightarrow] \\
    \sym^2 \H^0(\omega_X) \arrow[r,twoheadrightarrow] &  \H^0(\omega_X^{\otimes 2})^\iota \arrow[r,hookrightarrow] & \H^0(\omega_X^{\otimes 2})
  \end{tikzcd}
\end{equation}
 
\begin{lemma}\label{lem:reconstruction--Lbar}
  The curve $X$ and its canonical image $\overline{X}$ fit into the diagram
  \begin{equation}\label{eq:reconstruction} 
  \begin{tikzcd}
    X\arrow[dr, phantom, "\ulcorner", very near start] \arrow[r,"\cl"] \arrow[d,"2:1"'] & \Pp W \arrow[d,"2:1"] \arrow[r,"v_2",hookrightarrow] & \Pp \sym^2 W \arrow[ld,dashed]\\
    \overline{X} \arrow[r, "q_\varepsilon"] & \Pp (\sym^2 W)^\iota
  \end{tikzcd}
\end{equation}
where $v_2$ is the second Veronese embedding and the dashed arrow is the projection map.
\end{lemma}
\begin{proof}
  The diagram \eqref{eq:reconstruction} refines \eqref{eq:even_bi--X}. Use the notation $Y$ and $\overline{Y}$ introduced there. By definition of $\cl$ as the pullback of $\co_Y(1)$, there is a canonical identification of $Y$ with $\Pp W$. In diagram \eqref{eq:even_bi--X} the vertical map $Y \to \overline{Y}$ is given by the quotient of $Y$ via $\iota$. Therefore, the pullback of $\co_{\overline{Y}}(1)$ to $Y$ identifies $\H^0(\co_{\overline{Y}}(1))$ with the $\iota$-invariant sections of $\co_Y(2)$. As a result, $\overline{Y}$ is canonically identified with $\Pp (\sym^2 W)^\iota$ and the map $Y \to \overline{Y}$ factors through the second Veronese.
\end{proof}

  We may now write $\overline{Y}$ for the projective line $\Pp(\sym^2 W)^\iota$ and $Y$ for the projective line $\Pp W$. Denote the two branch points of the double cover $Y \to \overline{Y}$ by $p_1$ and $p_2$.

\begin{lemma}
  The image of the symmetrization $\ca_\varepsilon$ coincides with $T$. More precisely, the space $(\sym^2 W)^\iota$ is the image of the following composition of maps
  \[
  \pushQED{\qed}
    \H^0(\omega_C)^\vee \overset{\ca_\varepsilon}{\too} \sym^2 \H^0(\omega_X) \to \co_{\overline{X}}(2) \isoto \H^0(\omega_X^{\otimes 2})^\iota. \qedhere
  \popQED
  \]
\end{lemma}

\begin{proof}[Proof of Theorem~\ref{thm:reconstruction} for Case 4]
  Let $E$ be the pull back of the discriminant locus in $|\sym^2 \H^0(\omega_X)|$ to $|T|$. Then $\Gamma_{(X,\cl,\cl)}$ is a cone over $E$. Thus there is an isomorphism $\Gamma_\varepsilon \simeq \Gamma_{(X,\cl,\cl)}$ but we do not have a canonical isomorphism at the moment. Nevertheless, there is a quadric in $|\MUV|$ cutting out $C$ from $\Gamma_{(X,\cl,\cl)}$. We will show that this quadric is $Q_{(X,\cl,\cl)}$. 

  Lacking a canonical identification of the ambient spaces of $\Gamma_\varepsilon$ and $\Gamma_{(X,\cl,\cl)}$ we need to use the plane $|T|$ instead. Lemma~\ref{lem:even_bi_Q} implies that the projection of $Q_C$ from the vertex of $\Gamma_\varepsilon$ realizes $Q_C$ as a double cover of $|T|$ branched over two lines, say $L_1$ and $L_2$, which are dual to the two points $p_1$ and $p_2$ in $\overline{Y} \subset \Pp T$. We need to show $\Gamma_{(X,\cl,\cl)}$ has the same property.

There exists a $(+1)$-eigenvector $w_1 \in W$ and a $(-1)$-eigenvector $w_2 \in W$.  The map $Y \to \overline{Y}$ is the quotient of $Y = \Pp W$ by $\iota$ with ramification points that are dual to the points $[w_1]$ and $[w_2]$ in $|W|$. The corresponding branch points $p_1,p_2 \in \overline{Y}$ are then dual to the points $r_1=[w_1^2]$ and $r_2=[w_2^2]$. That is, the lines $L_1$ and $L_2$ are obtained by pulling back $r_1$ and $r_2$ via the projection $|T| \ratto |(\sym^2 W)^\iota|$.

The quadric cone $Q_{(X,\cl,\cl)}$ is obtained by pulling back the smooth quadric $|W|\times |W|$ in $|W \otimes W|$ to $|\MUV|$. This pullback factors through the plane $|(W\otimes W)^\iota|$ and $|W| \times |W|$ intersects this plane to give a conic $q$ over which $Q_{(X,\cl,\cl)}$ is a cone. By construction, the conic $q$ is the locus of $\iota$-invariant rank-1 tensors and can be described as the image of the following map:
\begin{equation}
  |W| \to |(W\otimes W)^\iota| \colon  [w] \mapsto [w\otimes \iota(w)].
  \label{eq:reconstruct-even-bi--parametrization}
\end{equation}

In order to determine the ramification locus of $Q_{(X,\cl,\cl)} \to |T|$, it is sufficient to determine the ramification points of the projection of $q$ down to $|(\sym^2 W)^\iota|$. Using the parametrization \eqref{eq:reconstruct-even-bi--parametrization} of $q$ we can write the projection map as follows:
  \[
    |W| \isoto q \tos |(\sym^2 W)^\iota| \colon  [w] \mapsto [w\cdot \iota(w)].
  \]
  Clearly, the two $\iota$-invariant points $[w_1],[w_2] \in |W|$ are the ramification points. The branch points are $r_1 = [w_1^2]$ and $r_2 = [w_2^2]$. Just as $Q_{(X,\cl,\cl)}$ is the cone over $q$, the branch locus of $Q_{(X,\cl,\cl)} \to |T|$ is the cone over $r_1$ and $r_2$, which is $L_1 \cup L_2$.
\end{proof}

\subsection{Descending $Q_{X,\kappa}$ and $\Gamma_{X,\kappa}$ to the ground field} \label{sec:descend_Q_Gamma}

The following result explains why the constructions given in previous subsections descend to the base field~$k$.
\begin{proposition}\label{P:QGammadescent}
	$X$ be a curve of genus three and let $\kappa\in \kum_X(k)\setminus \{0\}$. Then $Q_{X,\kappa}$ and $\Gamma_{X,\kappa}$ are defined over $k$.
\end{proposition}
\begin{proof}

  Let $\kappa$ correspond to the unordered pair of classes $\{[\cl_1], [\cl_2]\}$ of line bundles, where $\cl_2$ is identified with $\omega_X \otimes \cl_1^\vee$. By hypothesis, $\kappa$ is defined over $k$, so each line bundle class $[\cl_i]$ is defined over at most a quadratic extension $k'$ of $k$. We write $\bs_i=\bs_{[\cl_i]}$ for the Brauer--Severi variety defined over $k'$ corresponding to $[\cl_i]$. The Brauer--Severi variety corresponding to $\omega_X^{\otimes 2}$ is a $\Pp^5$. Divisor addition is represented by a morphism of $k'$-varieties
  \[
   \xa\colon \bs_1 \times \bs_2 \to \Pp^5.
  \]
  Observe that the action of $\Gal(k'/k)$ is limited to swapping $\bs_1,\bs_2$, so the image of $\xa$ is Galois-invariant and therefore defined over $k$. In fact, if $k'\neq k$ then we can descend $\xa$ to $k$ by replacing $\bs_1\times \bs_2$ with the Weil restriction of scalars of $\bs_1$ relative to $k'/k$.
  
 First assume $X$ is not hyperelliptic and the pencils are not self-residual. Then the image of $\xa$ is a quadric surface $Q_{X,\kappa}$ which spans a linear projective 3-space $\ppp$. The identification of $\Pp^5 = |\omega_X^{\otimes 2}|$ with $|\sym^2 \H^0(\omega_X)|$ gives us the discriminant locus $\cd$ which intersected with our $\ppp$ gives $\Gamma_{X,\kappa}$.

 Second, assume $X$ is not hyperelliptic but the pencils are self-residual, so $[L_1]=[L_2]$. In particular, the Brauer--Severi varieties $\bs_1=\bs_2$ are defined over $k$. The map $\xa$ is defined over $k$ and its image is a $\pp$ in $\Pp^5$ over $\kbar$ and therefore over $k$. We denote the branch locus with $q_{\xa}$, which is defined over $k$ and is the image of the diagonal of $\bs_1 \times \bs_1$. We identify $\Pp^5$ with $|\sym^2 \H^0(\omega_X)|$ once again to obtain a plane cubic $E \subset \pp$ as the pullback of the discriminant locus. We can take $Q_{X,\kappa}$ as a double cover of $\PP^2$ branched over $q_\xa$. Then $Q_{X,\kappa}$ admits a smooth quadric model in $\PP^3$, the cover $Q_{X,\kappa}\to\PP^2$ extends to a projection $\PP^3\dashrightarrow\PP^2$ and $\Gamma_{X,\kappa}$ is the cone over $E$. There is choice here: the double cover $Q_{X,\kappa}\to\PP^2$ admits quadratic twists. We can normalize it to have square discriminant, so that its rulings are defined over $k$. It is straightforward to check that taking a quadratic twist of $Q_{X,\kappa}$ corresponds to taking a quadratic twist of the double cover $C\to E$.

 Third, assume $X$ is hyperelliptic but the pencils are not self-residual. As in the first case, we find a quadric surface $Q' \subset \ppp_k \subset |\H^0(\omega_X^{\otimes 2})|$. However, this quadric surface is not $Q_{X,\kappa}$. We have a projection $|\sym^2\H^0(\omega_X)| \ratto |\H^0(\omega_X^{\otimes 2})|$ and the pullback of the $\ppp_k$ above is another projective 3-space $|\MUV| \subset |\sym^2 \H^0(\omega_X)|$. In $|\MUV|$ we find our $\Gamma_{X,\kappa}$ as usual by the pullback of the discriminant locus. The pullback of $Q'$ to $|\MUV|$ is our $Q_{X,\kappa}$ which is a singular quadric cone.

 Finally, assume $X$ is hyperelliptic and the pencils are self-residual. This is a combination of the last two constructions. The image of $\xa$ is a rational $\pp_k$ with branch locus a conic defined over $k$. As in the previous case, pull these back to $|\sym^2 \H^0(\omega_X)|$ via the projection to $|\H^0(\omega_X^{\otimes 2})|$. The pullback of the plane $\pp_k$ is another plane $\pp_k$ but the pullback of the quadric branch locus is a pair of lines. Now this latter $\pp_k$ contains also a cubic plane curve $E$ as the discriminant locus. Proceed as in the second case to find the cone $\Gamma_{X,\kappa}$ over $E$ and a double cover $Q_{X,\kappa}$ of $\pp_k$ branched over the two lines.  
\end{proof}

\section{Milne's correspondence between tritangents and bitangents}\label{sec:classical}

Let $C \subset \ppp$ be a \emph{generic} canonical curve of genus four with a line bundle $\varepsilon$ of order two. We drop the subscripts and denote the quadric $Q_C$ containing $C$ by $Q$, the plane quartic $X_\varepsilon$ associated to $(C,\varepsilon)$ by $X$. The goal of this section is to show that there is a correspondence between bitangents of $X$ and pairs of tritangents of $C$ differing by $\varepsilon$ where two tritangents $H_1$ and $H_2$, $H_i \cdot C=2D_i$, are said to \emph{differ by $\varepsilon$} if $D_1-D_2 \equiv \varepsilon$.

Let $\Gamma$ be the cubic symmetroid corresponding to $\varepsilon$ and $\xq\colon  \pp \to \vppp$ its symmetrization, and $\xc \colon  \pp \ratto \Gamma \subset \ppp$ its adjugate (Section~\ref{sec:adjugation_map}). A line $L \subset \pp$ will be called \emph{generic} if it avoids the base locus of $\xc$.  For generic $L$, $\xq(L)\subset \vppp$ is a conic and $\xc(L) \subset \ppp$ is a twisted cubic. As $C$ is generic, we assume the bitangents of $X$ are generic (Remark~\ref{rem:generic_bitangents}). For a generic line $L \subset \pp$ let $\Lambda_L$ be the enveloping cone of the family of planes defined by $\xq(L) \subset \vppp$. Let $T_L$ stand for the twisted cubic $\xc(L) \subset \Gamma$. We give a modern proof, outlined in Section~\ref{sec:milnes_proof}, of the following statement.

\begin{theorem}[Milne~\cite{Milne1923}]\label{thm:milne}
  Let $L \subset \pp$ be a bitangent of the curve $X$. The enveloping cone $\Lambda_L$ intersects $Q$ along two conics. The two planes spanned by these conics are a pair of tritangents of $C$ that differ by $\varepsilon$.
\end{theorem}

\noindent Starting from a pair of tritangents $H_1$ and $H_2$ of $C$ differing by $\varepsilon$, one can reverse the construction to obtain a bitangent of $X$ as follows. Let $D_1$ and $D_2$ be divisors on $C$ such that $H_i \cdot C = 2D_i$. Since $C$ is generic we assume that each $D_i$ is supported on three points.

\begin{corollary}\label{cor:milne_cor1}
 There is a twisted cubic $T$ passing through the six points in the support of $D_1$ and $D_2$. The twisted cubic $T$ is contained in $\Gamma$ and its preimage $\xc\inv(T)$ is a bitangent line of $X$. 
\end{corollary}
\begin{proof}
  Identify the lines in the ambient space $\pp$ of $X$ with divisors in $|\omega_C \otimes \varepsilon|$ (Theorem~\ref{thm:main_prym}). Let $L \subset \pp$ be the line corresponding to the divisor $D_1+D_2$. We show in Corollary~\ref{cor:DLinTL} that the twisted cubic $T_L = \xc(L)$ in $\Gamma$ passes through the points in the support of $D_1+D_2$. This line must be a bitangent of $X$ by Proposition~\ref{prop:bits_trits}.
\end{proof}

In fact, this proof implies that if we view a generic divisor $D \in |\omega_C \otimes \varepsilon|$ as six points on the canonical curve $C \subset \ppp$, then we can find a twisted cubic $T \subset \ppp$ passing through $D \subset C$. Moreover, this twisted cubic can be used to reconstruct $\Gamma$ as follows.

\begin{corollary}\label{cor:milne_cor2}
  Let $T$ be the twisted cubic passing through the six points in $C$ defining a generic divisor $D \in |\omega_C \otimes \varepsilon|$. Then, there is a unique cubic hypersurface containing $C$ and $T$---it is the Cayley cubic $\Gamma$.
\end{corollary}
\begin{proof}
  We follow Milne's original argument here. Any cubic surface containing ten points on $T$ will contain $T$ by B\'ezout theorem. A cubic surface containing $C$ will already contain the six points of contact which lie on $T$. Fix four general points $t_1,\dots,t_4$ on $T$ and consider the locus of cubic surfaces containing $C$ as well as $t_1,\dots,t_4$. Since the space of cubics containing $C$ is a projective space of dimension four, we expect to find a unique cubic surface satisfying these conditions. This expectation is realized when $C$ and $t_1,\dots,t_4$ are generic, as can be readily checked. We are done since $\Gamma$ contains both $T$ and $C$ by Corollary~\ref{cor:milne_cor1}.
\end{proof}

\begin{remark}\label{rem:generic_bitangents}
  We claim that if $C$ is generic, then the divisor $D_1 + D_2$ defined by a pair of tritangent planes of $C$ will be generic in the sense of this result. It suffices to generate one such $C$. To construct it, start with four lines $L_1,\dots,L_4$ in general position and construct $\xc\colon \pp \ratto \ppp$ using the cubics passing through the six points $L_i \cap L_j$, $i,j=1,\dots,4$. A generic quartic plane curve $X$ clearly has bitangents avoiding these six points. To construct $C$ from $X$, form the adjugate map $\xq\colon \pp \to \vppp$, choose a general quadric $\vQ$ containing $\xq(X)$ and intersect the dual of $\vQ$ with $\xc(\pp)$.
\end{remark}

\subsection{Outline for the proof of Milne's theorem}\label{sec:milnes_proof}

  A quadric surface $Q' \subset \ppp$ is said to be a \emph{sextactic quadric} of $C$ if $Q' \cdot C = 2D$. If $D \in |\omega_C \otimes \varepsilon|$ then $Q'$ is said to \emph{belong to the system corresponding to $\varepsilon$}. There are 255 systems, each corresponding to an order two point of $\jac_C[2]$. Each system is a three dimensional cone over a second Veronese image of $\pp$ with vertex $Q$.

  As usual $\rho_\varepsilon\colon C \to \pp$ denotes the Prym-canonical map (see also Theorem~\ref{thm:main_prym}). For a line $L \subset \pp$ let $D_L = \rho_\varepsilon^*(L) \in |\omega_C \otimes \varepsilon|$. The set of quadrics in $\ppp$ intersecting $C$ along $2D_L$ span a pencil containing $Q$. Denote this pencil by $\cp_L \subset |\co_{\ppp}(2)|$. 
  
Using Proposition~\ref{prop:DL} we distinguish an element of $\cp_L$: For a generic line $L \subset \pp$ the envelope $\Lambda_L$ of the conic $\xq(L) \subset \vppp$ is contained in $\cp_L$. Now we wish to determine the lines $L$ for which $\cp_L=\langle \Lambda_L,Q \rangle$ contains a reducible quadric.
  
\begin{proposition}\label{prop:bits_trits}
  The pencil $\cp_L$ contains a reducible quadric $H_1 \cup H_2$ if and only if $L$ is a bitangent of $X$. In this case, the planes $H_i$ are tritangents of $C$.
\end{proposition}
\begin{proof}
  Lemma~\ref{lem:cone_lemma} implies that $\langle \Lambda_L ,Q \rangle$ contains a reducible quadric if and only if the locus of enveloping tangent planes $\xq(L)$ of $\Lambda_L$ is bitangent to the dual quadric $\vQ \subset \vppp$. By Theorem~\ref{thm:main_prym} we know that $X$ equals $\xq\inv(\vQ)$. Therefore, the curve $\xq(L)$ is bitangent to $\vQ$ if and only if $L$ is bitangent to $X$.  It is clear that $H_1$ and $H_2$ cannot be equal, giving a non-reduced element in $\cp_L$. Otherwise, $D_L$ would belong to the canonical linear system $|\omega_C|$ instead of $|\omega_C \otimes \varepsilon|$. Now apply Lemma~\ref{lem:reducible_is_tritangent} to conclude that the planes $H_i$ are tritangents to $C$.
\end{proof}

These argument above completes the proof of Theorem~\ref{thm:milne} modulo the proofs of Proposition~\ref{prop:DL}, Lemma~\ref{lem:cone_lemma} and Lemma~\ref{lem:reducible_is_tritangent}. These proofs appear in Sections~\ref{sec:classical_cayley} and~\ref{sec:cone_lemma}.

\subsection{Removing the genericity assumptions}

In this subsection we will give, without proof, an indication as to what needs to be changed if genericity assumptions are to be removed.
  
  First, suppose that $(C,\varepsilon)$ is not bielliptic and the quadric $Q$ containing $C$ is smooth. Assuming no other restrictions on $(C,\varepsilon)$, the bitangents of $X$ may pass through some of the base points of the parametrization $\xc\colon \pp \ratto \Gamma$. Depending on how many base points are contained in a particular bitangent $L \subset \pp$ the corresponding ``twisted cubic'' $T_L$ may degenerate to any one of the following: a union of a conic and a line intersecting it at a point, a union of three concurrent non-planar lines, a union of two skew lines and a third line intersecting the first two. 
  
  Except for finitely many lines $L$, the image $\xq(L)$ is a conic and the correct choice of $\Lambda_L$ is its enveloping cone. However, for finitely many $L$, the map $L \overset{\xq}{\to} \xq(L)$ will be 2-to-1 and $\xq(L)$ will be a line. In this case, the correct choice of $\Lambda_L$ is the union of the two planes dual to the branch points on $\xq(L)$. These two planes are already tritangents---unless the pencil $\langle \Lambda_L, Q \rangle$ contains another pair of planes, in which case, the latter pair of planes will be tritangents. 

  When $Q$ is singular, we must take $(C,\varepsilon)$ to be even since an odd pair $(C,\varepsilon)$ does not behave well from the very start (e.g.\ see Corollary~\ref{cor:odd_is_evil}). When $(C,\varepsilon)$ is even, the canonical image of $X$ is a conic $\overline{X} \subset \pp$. The $28$ ``bitangents'' of $X$ will then be the 28 lines passing through any two of the images of the eight Weierstrass points in $\overline{X}$. The images $\xq(L)$ of these bitangent lines are always conics and the correct choice of $\Lambda_L$ is again the envelope of $\xq(L)$.

\subsection{Cayley cubic and enveloping cones}\label{sec:classical_cayley} 

Let $n_1,n_2,n_3,n_4$ denote the four nodes of the Cayley cubic $\Gamma$ and $\overline{E}_{ij} := \langle n_i,n_j \rangle \subset \Gamma$ the lines between them. Let $\psi\colon S \to \Gamma$ be the resolution of the four nodes by blow-ups. We will write $N_i \subset S$ for the preimage of the node $n_i$ and $E_{ij}$ for the proper transform of $\overline{E}_{ij}$. The $E_{ij}$'s are disjoint $(-1)$-curves and their blow-down gives a map $\pi\colon S \to \pp$.  The images of $N_i$ in $\pp$ are lines, denoted by $\overline{N}_i$. The curves $E_{ij}$ collapse to points $e_{ij}$.  The line class in $\pp$ as well as its pullback $\pi^* \ell$ in $S$ will be denoted by $\ell$.  The following result is straightforward and well known so we state it without proof.

\begin{lemma}\label{lem:anti-S}\label{lem:StoP3}
 For all $i=1,\dots,4$ we have $N_i \equiv \ell - \sum_{j \neq i} E_{ij}$. Also $\omega_S \equiv -3\ell + \sum_{i < j} E_{ij}$. Finally, the map $\psi\colon S \to \ppp$ is induced by the anti-canonical linear system $\vomega_S$. \qed
\end{lemma}

From Lemma~\ref{lem:StoP3} we see that any quadric hypersurface $Q' \subset \ppp$ pulls back on $S$ to a divisor in $|\vomega_S^{\otimes 2}|$. Since $C$ misses the nodes of $\Gamma$, we identify $C$ and its pullback to $S$.

\begin{lemma}
  The pullback map $|\co_{\ppp}(2)| \to |\vomega_S^{\otimes 2}| : Q' \mapsto \psi^*(Q')$ is an isomorphism.
\end{lemma}
\begin{proof}
  The pullback map is obtained from the restriction map $\H^0(S,\vomega_S^{\otimes 2}) \to \H^0(\ppp,\co_{\ppp}(2))$. Since $\Gamma$ is not contained in a quadric, this map has trivial kernel. As $h^0(\ppp,\co_{\ppp}(2))=10$, it remains to show $h^0(S,\vomega_S^{\otimes 2})=10$. Since $C$ is a quadric section we have $C \equiv \vomega_S^{\otimes 2}$. By adjunction we get $\vomega_S|_C \simeq \omega_C$. This yields the following exact sequence:
$  \begin{tikzcd}[cramped, sep=small] 0 \arrow[r] & \co_S \arrow[r] & \vomega_S^{\otimes 2} \arrow[r] & \omega_C^{\otimes 2} \arrow[r] & 0 \end{tikzcd}$.  As $S$ is rational, the irregularity of $S$ is 0. Taking global sections of this sequence gives that $h^0(\vomega_S^{\otimes 2})= h^0(\omega_C^{\otimes 2})+1$. By Riemann--Roch on $C$ we have $h^0(\omega_C^{\otimes 2})=9$.
\end{proof}

\begin{lemma}\label{lem:cone_through_nodes}
  Given any line $L \subset \pp$, there is a unique quadric $\Lambda$ in $\ppp$ such that $\psi^*(\Lambda) \equiv 2\pi^*(L) + N_1+N_2+N_3+N_4$.
\end{lemma}
\begin{proof}
  This is an immediate consequence of the previous lemma, combined with the observation that  $2\pi^*(L) + N_1+N_2+N_3+N_4 \equiv \vomega_S^{\otimes 2}$.
\end{proof}

\begin{proposition}\label{prop:envelope}
  For a line $L \subset \pp$ not passing through any of the points $e_{ij}$, there exists a unique quadric $\Lambda$ passing through the nodes of $\Gamma$ and tangential to $\Gamma$ all along the twisted cubic $\xc(L)$. This quadric is the envelope $\Lambda_L$ of $\xq(L)$.
\end{proposition}
\begin{proof}
  Since $L$ does not pass through $e_{ij}$, the map $\xc\colon \pp \ratto \Gamma$ maps $L$ to a twisted cubic.  Let $\Lambda$ be as in Lemma~\ref{lem:cone_through_nodes}. By degree reasons, we have $\Gamma \cdot \Lambda = 2 \xc(L)$. However, the pullback of a degree three hypersurface onto a smooth quadric is a $(3,3)$-class, which is not divisible by two. Therefore, $\Lambda$ must be singular. Then, $\Lambda$ must be an irreducible quadric because a twisted cubic does not lie on a plane.  A cone is defined by its enveloping tangent planes. For each point on the twisted cubic $\xc(L)$, the tangent plane of $\Lambda$ coincides with the tangent plane of $\Gamma$. With $\gamma\colon \ppp \ratto \vppp$ the Gauss map of $\Gamma$ we then conclude that the enveloping planes of $\Lambda$ is the image $\gamma\circ\xc(L)=\xq(L)$ (Lemma~\ref{lem:gauss}).
\end{proof}

\begin{lemma}\label{lem:blow_down_is_prym}
 The restriction of the blow-down map $\pi\colon S \to\pp$ to $C$ coincides with the Prym-canonical map $\rho_\varepsilon \colon C \to \pp$ induced by the line bundle $\omega_C \otimes \varepsilon$. 
\end{lemma}
\begin{proof}
  By Theorem~\ref{thm:induced_double_cover} the double cover of $\Gamma$ branched over the nodes induces the double cover on $C$ corresponding to $\varepsilon$. The preimage of the nodes in $S$ gives a divisor divisible by two: $N_1+N_2+N_3+N_4 \equiv 2(2L - \sum_{i<j} E_{ij})$. Since $S$ is simply connected, there is a unique divisor class whose double is $N_1+\dots+N_4$. Therefore, there is a unique double cover of $S$ branched over the preimage of the nodes of $\Gamma$. This double cover resolves the singularities of the double cover of $\Gamma$ branched over the nodes~\cite{catanese83}. Therefore, the intersection $(2L - \sum_{i<j} E_{ij}) \cdot C$ must give a divisor class equivalent to $\varepsilon$.  Since $C$ is a quadric section of $S$ we have $C \equiv \vomega_S^{\otimes 2}$ by Lemma~\ref{lem:StoP3}. By adjunction we get $(3\ell - \sum_{i<j}E_{ij})|_C \equiv \omega_C$. As a result $\omega_C \otimes \varepsilon \equiv \ell|_C$.
\end{proof}

\begin{corollary}\label{cor:DLinTL}
  Let $D \in |\omega_C\otimes \varepsilon|$ be a general divisor and $L \subset \pp$ be the line for which $\rho_\varepsilon^*(L)=D$. Then the twisted cubic $\xc(L)$ passes through the six points in the support of $D$.
\end{corollary}
\begin{proof}
  Since $D$ is general $\xc(L)$ is a twisted cubic. The cubic map $\xc \colon \pp \ratto \Gamma$ is the composition of the rational inverse of $\pi$ with $\psi\colon S \to \Gamma$. Thus $\xc(L) = \psi(\pi\inv(L))$. We are done by Lemma~\ref{lem:blow_down_is_prym} since $\pi\inv(L) \cdot_S C = D$.
\end{proof}

\begin{proposition}\label{prop:DL}
 The envelope $\Lambda_L$ satisfies $\Lambda_L \cdot C = 2D_L$.
\end{proposition}
\begin{proof}
  By Proposition~\ref{prop:envelope} and Lemma~\ref{lem:cone_through_nodes} we have $\Lambda_L|_S \equiv 2\pi^*(L) + N_1+N_2+N_3+N_4$. Since $C$ does not pass through the nodes we get $N_i \cdot C = 0$. Therefore, $\Lambda_L|_S \cdot C \equiv 2\pi^*(L) \cdot C \equiv 2 \rho_\varepsilon^*(L)=2D_L$ where we used Lemma~\ref{lem:blow_down_is_prym}. 
\end{proof}

\subsection{Pencils of quadrics}\label{sec:cone_lemma}

Let $\Lambda \subset \ppp$ be an irreducible quadric cone and $Q \subset \ppp$ be a smooth quadric. Denote the dual of $Q$ by $\vQ \subset \vppp$. The dual of the cone $\Lambda \subset \ppp$ is a plane $\vLam$ and there is a conic $\vlam \subset \vLam$ parametrizing the enveloping tangent planes of $\Lambda$.

The following lemma reformulates, in terms of $\vlam$ and $\vQ$, the condition for $\Lambda$ and $Q$ to intersect along two conics. Clearly the base locus $\Lambda \cap Q$ of the pencil $\langle \Lambda,Q \rangle$ breaks into a union of two conics if and only if $\langle \Lambda,Q \rangle$ contains a reducible conic. We will say $\vlam$ is \emph{bitangent} to $\vQ$ if $\vlam \not\subset \vQ$ and $\vlam$ is tangential to $\vQ$ at every point of the intersection $\vlam \cap \vQ$. 

\begin{lemma}\label{lem:cone_lemma}
  The pencil $\langle \Lambda,Q \rangle$ contains a reducible but reduced quadric if and only if $\vlam$ is bitangent to $\vQ$. The pencil $\langle \Lambda,Q \rangle$ contains a non-reduced element if and only if $\vlam \subset \vQ$.
\end{lemma}

\begin{proof}[Proof of Lemma~\ref{lem:cone_lemma}]
  Let $x\in \Lambda$ be the vertex of the cone and $\pi_x\colon \ppp \ratto \pp_x$ be the projection map from $x$. The plane $\vLam$ is dual to $\pp_x$ and $\pi_x(\Lambda)$ is a conic $\lambda \subset \pp_x$ dual to $\vlam$.

  If the vertex $x$ is not on $Q$ then the projection map $\pi_x$ realizes $Q$ as a double cover of the plane $\pp_x$. When $x$ is on $Q$ then the projection maps $Q$ birationally to $\pp_x$, contracting the two lines of $Q$ through $x$. In either case, we will write $R \subset Q$ for the ramification locus in $Q$ of the projection $\pi_x|_Q$---this is a smooth conic if $x \notin Q$ and a pair of lines through $x$ when $x\in Q$. Denote the image of $R$ in $\pp_x$ by $B$, which is the branch locus of $\pi_x|_Q$. Note that $B$ is a smooth conic if $x \notin Q$ and $B$ is a pair of distinct points if $x \in Q$. 
  
  The intersection $\vLam \cap \vQ$ is dual to $B$. Indeed, $\vLam \cap \vQ$ parametrizes the tangent planes of $Q$ which pass through the vertex $x \in \Lambda$.  A point $p \in Q$ is in $R$, and thus maps to a point in $B$, if and only if the ray $\langle x,p \rangle$ is contained in the tangent plane of $Q$ at $p$.

  \textbf{Case 1:} We have $H_1 = H_2$ if and only if the intersection $\Lambda \cap Q$ is non-reduced everywhere. This happens precisely when the enveloping planes of $\Lambda$ are tangent to $Q$, which means $\vlam \subset \vQ$. 

  \textbf{Case 2:} Now let us assume $x \in Q$. Notice that $H_1 \neq H_2$ since $\vlam$ is a conic but $\vLam \cap Q$ is a pair of lines so that $\vlam$ is not contained in $\vQ$. In this case, $B$ consists of two distinct points, whose dual is the union of the two lines $\vLam \cap \vQ$. The conic $\vlam$ is tangential to these two lines $\vLam \cap \vQ$ if and only if its dual $\lambda$ contains $B$.

The two points in $B$ are the images of the two lines in $Q$ passing through $x$. We have $B \subset \lambda$ if and only if these two lines of $Q$ are contained in $\Lambda$. It is readily seen that the intersection $Q \cap \Lambda$ splits into the union of two hyperplane sections of $Q$ precisely when one of these hyperplanes is the tangent plane of $Q$ at $x$, which is the case here.

  \textbf{Case 3:} We now assume $x\notin Q$ and $H_1 \neq H_2$, this is the generic case. Since $x \notin Q$ the intersection $\vQ \cap \vLam$ is a smooth conic and then so is its dual $B$. The conic $\vlam$ is a bitangent of $\vQ$ if and only if $\vlam$ is a bitangent of $\vLam \cap \vQ$. The latter holds if and only if $\lambda$ and $B$ are bitangent.

  Consider a point $\bar y \in B$, which has a unique preimage $y \in Q$. By definition of $B$, the tangent plane $\mathrm{T}_yQ$ of $Q$ at $y$ passes through $x$. Furthermore, since $R$ maps isomorphically to $B$, and because the tangent direction of $R$ at $y$ lies in $\mathrm{T}_yQ$, the tangent line of $B$ at $\bar y$ is $\pi_x(\mathrm{T}_yQ)$. Thus, $\lambda$ and $B$ are tangential at $\bar y$ if and only if the tangent planes of $\Lambda$ and $Q$ coincide at $y$. 
  
  We further note that, since $x \notin Q$, the intersection $\Lambda \cap Q$ is singular at a point $y$ if and only if the tangent planes of $Q$ and $\Lambda$ coincide at $y$. In other words, $\Lambda \cap Q$ is singular at $y$ if and only if $\lambda$ and $B$ are tangential at $\bar y$.  On the other hand, $\Lambda \cap Q \to \lambda$ is ramified exactly at the points of intersection $\lambda \cap B$. A point in $\Lambda \cap Q$ which lies over a point of transversal intersection in $\lambda \cap B$ is an honest ramification point of a smooth component of $\Lambda \cap Q$ over $\lambda$.  The intersection $\Lambda \cap Q$ splits into two conics if and only if there is no non-singular points in $\Lambda \cap Q$ that are ramified over $\lambda$. This happens precisely $\lambda$ intersects $B$ with even multiplicity at every point of contact, i.e., $\lambda$ is a bitangent of $B$.
\end{proof}

Let $\Gamma \subset \ppp$ be any normal cubic and $C = \Gamma \cap Q$ a smooth curve. We will assume that the intersection of $\Lambda$ with $\Gamma$ yields an effective curve of even multiplicity, that is $\Lambda \cdot \Gamma = 2 Y$.

\begin{lemma}\label{lem:reducible_is_tritangent}
  If there is a reducible and reduced element $H_1 \cup H_2$ in the pencil $\langle  \Lambda,Q \rangle$ then it is unique and $H_i$ are tritangents of $C$.
\end{lemma}
\begin{proof}
  At most one of the conic $\sigma_i = H_i \cap Q$ can be reducible. Indeed, any collection of lines in $\Lambda$ are concurrent at the vertex of $\Lambda$. However, since $H_1 \neq H_2$, if both $\sigma_i$ are reducible we will have at least three non-concurrent lines in $\Lambda$.  One of $H_1$ or $H_2$ is necessarily the span of the irreducible conic $\sigma_1$ or $\sigma_2$. The other plane is then the span of the other conic. This proves that the reducible element of the pencil is unique.

Let $D_i := \sigma_i \cdot_Q C$. Since $\Lambda \cdot \Gamma = 2\cdot Y$ we have $\Lambda \cdot C = D_1 + D_2 = 2 (Y \cdot_{\Gamma} C)$. That is, the sum $D_1 + D_2$ is an even divisor. In order to prove that $H_i$ is a tritangent of $C$ we need to show that $D_i$ is even. This condition could fail to hold only if $C$ passes through a point $p$ in $\sigma_1 \cap \sigma_2$ \emph{and} if $\sigma_i$ intersects $C$ transversally at $p$. We will prove that $\sigma_1 \cap \sigma_2 \cap C = \emptyset$.

  The cone $\Lambda$ may have its vertex on $Q$, this would be the singular point of the reducible conic $\sigma_i$. But there is no smooth conic passing through the vertex of $\Lambda$. Therefore, $\sigma_1 \cap \sigma_2$ is contained in the smooth locus of $\Lambda$.

  By hypothesis, $\Lambda$ is tangential to $\Gamma$ at smooth points of contact. The cone $\Lambda$ is tangential to $Q$ at the singular points of $\sigma_1 \cup \sigma_2$, in particular on $\sigma_1 \cap \sigma_2$. Since the intersection $C = \Gamma \cap Q$ is smooth, the tangent planes of $\Gamma$ and $Q$ must both be well defined and transversal at every point of $C$. Therefore, a point $p \in \sigma_1 \cap \sigma_2$ cannot lie on $C$, otherwise the tangent planes of $\Gamma$ and $Q$ at $p$ would coincide with the tangent plane of $\Lambda$ at $p$.
\end{proof}

\end{document}

%% file: preamble.tex
\usepackage[utf8]{inputenc}
\usepackage[T1]{fontenc} 
\usepackage{microtype}
\usepackage[american]{babel}
\usepackage{pdfpages}
\usepackage{emptypage}
\usepackage{geometry}

\usepackage{listings}
\usepackage{color}
\usepackage{xcolor}
\usepackage{listings}

\usepackage{caption}
\DeclareCaptionFont{white}{\color{white}}
\DeclareCaptionFormat{listing}{\colorbox{gray}{\parbox{\textwidth}{#1#2#3}}}
\usepackage[normalem]{ulem}

\usepackage{amsmath}
\usepackage{mathtools}
\usepackage{amsthm}
\usepackage{amssymb}
\usepackage{enumitem}

\usepackage{hyperref}

\usepackage{tikz-cd}
\usepackage{chngcntr}

\usepackage{amsfonts}
\usepackage{yfonts}
\usepackage{mathrsfs}
\usepackage[%
  cal=euler,
  bb=ams,
  scr=rsfso%
]{mathalfa} 

\usepackage{pdfcomment}

\usetikzlibrary{decorations.pathmorphing}

\DeclareMathOperator{\im}{im}

\DeclareMathOperator{\Spec}{\underline{S}p\underline{ec}}
\DeclareMathOperator{\spec}{Spec}

\DeclareMathOperator{\jac}{Jac}

\DeclareMathOperator{\Pic}{Pic}

\DeclareMathOperator{\rk}{rk}

\DeclareMathOperator{\pr}{pr}

\DeclareMathOperator{\sym}{Sym}

\DeclareMathOperator{\ann}{Ann}

\newcommand{\dd}{\mathrm{d}}

\DeclareMathOperator{\disc}{disc}

\DeclareMathOperator{\adj}{adj}

\usepackage{array}   % for \newcolumntype macro
\newcolumntype{M}{>{$}c<{$}} % math-mode version of "c" column type
\newcolumntype{L}{>{$}l<{$}}
\usepackage{arydshln}
\usepackage{multirow}

\DeclareMathOperator{\Hom}{\mathscr{H}\kern -2pt om}
\DeclareMathOperator{\Ext}{\mathscr{E}\kern -1.5pt xt}

\newtheorem{theorem}{Theorem}[section]
\numberwithin{theorem}{section} % Theorem number is of the form Section.Subsection.Thm

\newcommand{\nameofthm}{}

\newtheorem{genericThm}[theorem]{\nameofthm}

{%
  \begin{proof}$ $\newline
  \indent ($\implies$) 
}
{\end{proof}}

\newtheorem*{theorem*}{Theorem}

\newtheorem{lemma}[theorem]{Lemma}
\newtheorem{corollary}[theorem]{Corollary}
\newtheorem{proposition}[theorem]{Proposition}

\theoremstyle{definition}

\newtheorem*{definition*}{Definition}
\newtheorem{definition}[theorem]{Definition}
\newtheorem{notation}[theorem]{Notation}
\newtheorem{remark}[theorem]{Remark}

\newtheorem*{notation*}{Notation}

\numberwithin{equation}{section}
\newcommand{\nameofenv}{}
\newtheorem*{generic}{\nameofenv}

\newcommand{\nameOfNumEnv}{}
\newtheorem{genericNumbered}[theorem]{\nameOfNumEnv}

\makeatletter
\def\imod#1{\allowbreak\mkern10mu({\operator@font mod}\,\,#1)}
\makeatother

\newcommand{\p}{\Pp^1}
\newcommand{\pp}{\Pp^2}
\newcommand{\ppp}{\Pp^3}
\newcommand{\vppp}{\widehat{\mathbb{P}}^3}
\newcommand{\vpp}{\widehat{\mathbb{P}}^2}

\newcommand{\vQ}{\widehat{Q}}
\newcommand{\vq}{\widehat{q}}
\newcommand{\vLam}{\widehat{\Lambda}}
\newcommand{\vlam}{\widehat{\lambda}}
\newcommand{\vGam}{\widehat{\Gamma}}

\newcommand{\vomega}{\widehat{\omega}}
\newcommand{\vL}{\widehat{L}}

\newcommand{\inv}{^{-1}}

\newcommand*{\ratto}{\mathbin{\tikz [baseline=-0.25ex,-latex, dashed,->] \draw (0pt,0.5ex) -- (1.3em,0.5ex);}}

\newcommand{\too}{\longrightarrow}
\newcommand{\tos}{\twoheadrightarrow}
\newcommand{\toi}{\hookrightarrow}
\newcommand{\isoto}{\overset{\sim}{\to}}

\newcommand{\Cc}{\mathbb{C}}

\newcommand{\Pp}{\mathbb{P}}

\renewcommand{\H}{\mathrm{H}}

\newcommand{\xa}{\mathfrak{a}}
\newcommand{\xb}{\mathfrak{b}}

\newcommand{\xc}{\mathfrak{c}}

\newcommand{\xq}{\mathfrak{q}}

\newcommand{\ca}{\mathcal{A}}

\newcommand{\cd}{\mathcal{D}}

\newcommand{\cf}{\mathcal{F}}

\newcommand{\ck}{\mathcal{K}}
\newcommand{\cl}{\mathcal{L}}

\newcommand{\cm}{\mathcal{M}}
\newcommand{\cn}{\mathcal{N}}
\newcommand{\co}{\mathcal{O}}

\newcommand{\cp}{\mathcal{P}}
\newcommand{\cR}{\mathcal{R}}
\newcommand{\cs}{\mathcal{S}}
\newcommand{\ct}{\mathcal{T}}

\newcommand{\cv}{\mathcal{V}}

\newcommand{\sep}{^{\mathrm{sep}}}

\newcommand{\MUV}{V'}